\documentclass[11pt]{article}
\usepackage[margin=1in]{geometry}
\usepackage{amsmath,amssymb,amsthm,mathtools,bm,booktabs,array}
\usepackage[colorlinks=true,linkcolor=blue,citecolor=blue,urlcolor=blue]{hyperref}
\usepackage{microtype,graphicx,xcolor}
\graphicspath{{figures/}}
\newtheorem{theorem}{Theorem}[section]
\newtheorem{lemma}[theorem]{Lemma}
\newtheorem{proposition}[theorem]{Proposition}
\newtheorem{corollary}[theorem]{Corollary}
\theoremstyle{definition}\newtheorem{assumption}[theorem]{Assumption}
\theoremstyle{remark}\newtheorem{remark}[theorem]{Remark}
\newcommand{\eps}{\epsilon_0}
\newcommand{\muo}{\mu_0}
\newcommand{\curl}{\nabla\times}
\newcommand{\norm}[1]{\left\|#1\right\|}
\newcommand{\ip}[2]{\left(#1,#2\right)_{\mathcal T_h}}
\newcommand{\bd}[2]{\left\langle#1,#2\right\rangle_{\partial\mathcal T_h}}
\newcommand{\jump}[1]{\left[#1\right]}
\newcommand{\Et}{\mathcal E}
\newcommand{\Dh}{\mathcal D_h}
\newcommand{\R}{\mathbb R}
\newcommand{\llbracket}{\mathopen{[\![}}
\newcommand{\rrbracket}{\mathclose{]\!]}}
\title{A Hybridizable Discontinuous Galerkin Method for\protect\\ Time-Modulated Drude--Lorentz Media}
\author{Jeonghun J. Lee\\
{\small Department of Mathematics, Baylor University}\\
{\small Waco, TX 76706, USA}\\
{\small\href{mailto:jeonghun_lee@baylor.edu}{\nolinkurl{jeonghun_lee@baylor.edu}}}}
\date{October 15, 2026}
\begin{document}
\maketitle
\begin{abstract}
We formulate and analyze a hybridizable discontinuous Galerkin (HDG) method for Maxwell's equations coupled to four time-modulated Drude--Lorentz models: modulated in-coupling, modulated polarization density, modulated current density, and modulated effective mass. The use of electric displacement and model-dependent continuous material states gives a common treatment of temporal interfaces while preserving the distinct physical transmission rules. We derive energy and stability estimates, establish spatial and event-aligned time-error bounds, and separate time-integration error from error caused by misplaced switching times. Numerical experiments in three dimensions and in reduced modal systems confirm the predicted spatial, temporal, and switch-timing behavior, including the field-jump effect specific to the polarization-density model.
\end{abstract}

\section{Introduction}

Time-modulated dispersive media require a dynamical constitutive law rather than a frequency-domain permittivity with parameters simply replaced by functions of time. Dixon et al.\ \cite{Dixon2025} demonstrate that four modulation mechanisms with the same frozen Drude--Lorentz permittivity can produce different temporal scattering coefficients. The distinction is encoded in the material evolution and in the quantities continuous at a temporal interface.

HDG discretizations of the time-domain Maxwell equations include implicit formulations by Christophe, Descombes, and Lanteri and by Berardocco, Kronbichler, and Gravemeier, as well as an explicit formulation by Nehmetallah et al.~\cite{Christophe2018,Berardocco2020,Nehmetallah2020}. The two Berardocco formulations use, respectively, an electromagnetic diffusion equation and the first-order Maxwell equations. Time-harmonic Maxwell HDG methods provide further context for hybridization and trace-space design~\cite{NguyenPeraireCockburn2011,ChenQiuShiSolano2017}. Related frequency-domain HDG methods address nonlocal hydrodynamic Drude and generalized nonlocal optical response models~\cite{Li2017}.

DG methods for time-domain Maxwell systems with material ODEs have also been analyzed for dispersive media. Wang, Xie, and Zhang establish stability and spatial error estimates for cold-plasma, Debye, and Lorentz models; Xie, Wang, and Zhang study a space-time DG discretization; and Bokil et al. derive energy-stable DG schemes for Maxwell equations coupled to nonlinear optical material ODEs, including Lorentz response~\cite{WangXieZhang2010,XieWangZhang2014,Bokil2017}. These works establish a broad numerical foundation for dispersive Maxwell systems. The present contribution concerns a different issue: four time-modulation mechanisms with distinct temporal transmission rules, and the resulting energy and error analysis for an event-aligned Maxwell--material HDG scheme. Earlier research programs also explicitly considered time-domain HDG with local dispersion models~\cite{Nachos2016}.

The polarization equations introduce no additional spatial derivatives in the local models considered here. Their unknowns therefore require no additional skeleton variables. This makes the method a useful reference solver for time-modulated wave propagation and, eventually, inverse design. Computational advantages over explicit DG or other solvers must nevertheless be demonstrated at equal response accuracy.

For smooth time-dependent coefficients, Al Jarroudi and Grote recently proved optimal energy-norm convergence for semidiscrete finite element approximations of second-order hyperbolic equations with coefficients varying in space and time, using a time-dependent Ritz-like projection~\cite{AlJarroudiGrote2026}. Renner and Yousept analyze Maxwell equations with time-dependent permittivity and permeability, nonlinear electric--magnetic coupling, and a nonlinear Silver--M\"uller boundary condition~\cite{RennerYousept2026}. Their edge-element/piecewise-constant approximation with implicit Euler yields strong convergence on finite time intervals under their refinement hypotheses, and exponential decay under additional structural and dissipativity assumptions. These results complement the present quantitative HDG error estimates for auxiliary dispersive states with a constant background permittivity and permeability: the material coefficients may jump, and we explicitly separate event-aligned integration from switch-time displacement.

The representation of temporal interfaces introduces an error distinct from midpoint integration error. Section~\ref{sec:timing} quantifies the state error caused by switch-time displacement, shows when it limits midpoint convergence to first order, and identifies the resulting displacement of simple modal stability boundaries. A discontinuous electric constitutive reconstruction in model 2 requires a further distinction between state convergence and field convergence.

\paragraph{Contributions and organization.}
We formulate four continuous-state models with their transmission rules, derive model-specific energy balances, and analyze an HDG method with spatial and event-aligned time discretization. Variant B has a proved common-state rate of order $h^{k+1}$; variant A retains its baseline estimate. We quantify errors from displaced switches, including a model-2 field jump and a stability-classification example. Sections~\ref{sec:models}--\ref{sec:hdg} introduce the models and method; Sections~\ref{sec:error}--\ref{sec:time} establish the spatial and time estimates; Section~\ref{sec:timing} treats switch displacement; Section~\ref{sec:numerical} reports the experiments; and Section~\ref{sec:discussion} discusses scope and open problems.

\section{Models, transmission conditions, and continuous energy}\label{sec:models}
Let $\Omega\subset\R^3$ be a bounded polyhedron with a conforming tetrahedral mesh $\mathcal T_h$. We impose a perfectly electrically conducting (PEC) boundary,
\begin{equation}\label{eq:pec}
 \mathbf{n}\times \mathbf{E}=0\quad\text{on }\partial\Omega.
\end{equation}
A periodic cell with paired boundary facets is an alternative throughout. The initial data satisfy the boundary compatibility needed for the regularity used in the error estimate.

Write $s=\omega_p^2$ and $b=\omega_0^2$, where $\omega_p$ is the (possibly time-modulated) plasma frequency setting the oscillator strength, and $\omega_0$ is the Lorentz resonance frequency. Thus $s$ and $b$ are the corresponding squared frequencies. With fixed positive background constants $\eps,\muo$, Maxwell's equations are
\begin{equation}\label{eq:maxwell}
 \partial_t\mathbf{D}-\curl \mathbf{H}=0,\qquad \muo \partial_t\mathbf{H}+\curl \mathbf{E}=0,\qquad \mathbf{D}=\eps \mathbf{E}+\mathbf{P}.
\end{equation}
Sources can be added, but are omitted to isolate energy transfer caused by modulation. The model is isotropic, with one electric Lorentz pole and no magnetic dispersion.

\begin{assumption}\label{ass:coeff}
There are fixed global event times $0=\theta_0<\theta_1<\cdots<\theta_{N_{\rm ev}}<T=\theta_{N_{\rm ev}+1}$. On each $I_j=(\theta_j,\theta_{j+1})$, $s,b,\gamma$ are elementwise constant in space and belong to $W^{1,\infty}(I_j)$ on each element. Their one-sided limits exist at events and
\[
 0<s_*\le s\le s^*,\qquad 0<b_*\le b\le b^*,\qquad 0\le\gamma\le\gamma^*.
\]
The mesh resolves fixed spatial material interfaces. The spatial mesh is unchanged at temporal events. Integrals are evaluated exactly, or with quadrature exact for the polynomial products appearing in the scheme.
\end{assumption}
The main model comparison takes $b$ constant in time, as in the scattering calculations of \cite{Dixon2025}; allowing a prescribed piecewise smooth $b$ is a mathematical extension. For effective-mass modulation, keeping $b$ fixed is a different physical assumption from keeping a spring constant fixed. The chosen modulation path must be stated in every application.

The homogeneous Maxwell equations propagate $\nabla\cdot \mathbf{D}$ and $\nabla\cdot(\muo \mathbf{H})$ in the distributional continuous formulation. We prescribe compatible divergence data. The HDG method below does not claim exact preservation of the full Gauss constraints.

\subsection{Continuous-state formulations}
All fields and material variables below are real vector fields. Complex Fourier modes used in validation are understood through their real and imaginary parts. We refer to the four constitutive mechanisms as (1) modulated in-coupling, (2) modulated polarization density, (3) modulated current density with independent memory, and (4) modulated effective mass. In each model, the listed states are absolutely continuous inside $I_j$ and continuous across temporal events. These first-order equations define the discontinuous-coefficient model and we do not expand derivatives of a step coefficient into distributions.

\subsubsection{Model 1: modulated in-coupling}
This is the modulated in-coupling (oscillator-strength) mechanism distinguished in the time-varying Drude--Lorentz comparison of Dixon et al.~\cite{Dixon2025}.
With $\mathbf{J}=\partial_t\mathbf{P}$,
\begin{equation}\label{eq:m1}
 \partial_t\mathbf{P}=\mathbf{J},\qquad \partial_t\mathbf{J}+\gamma \mathbf{J}+b\mathbf{P}=\eps s\mathbf{E}.
\end{equation}
The continuous state is $Y_1=(\mathbf{D},\mathbf{H},\mathbf{P},\mathbf{J})$. At every event,
\begin{equation}\label{eq:jump1}
 \jump{\mathbf{D}}=\jump{\mathbf{H}}=\jump{\mathbf{P}}=\jump{\mathbf{J}}=0.
\end{equation}
Thus $\mathbf{E}$ is continuous. The energy is
\begin{equation}\label{eq:energy1}
 \Et_1=\frac12\int_\Omega\left(\eps|\mathbf{E}|^2+\muo|\mathbf{H}|^2+
 \frac{|\mathbf{J}|^2+b|\mathbf{P}|^2}{\eps s}\right)dx.
\end{equation}

\subsubsection{Model 2: modulated polarization density}
This is the modulated polarization-density mechanism of Dixon et al.~\cite{Dixon2025}, in which the scaled polarization $\mathbf{Q}=\mathbf{P}/s$ is the continuous material coordinate across a switch.
Define $\mathbf{Q}=\mathbf{P}/s$ and $\mathbf{V}=\partial_t\mathbf{Q}$. The material dynamics and constitutive law are
\begin{equation}\label{eq:m2}
 \partial_t\mathbf{Q}=\mathbf{V},\qquad \partial_t\mathbf{V}+\gamma \mathbf{V}+b\mathbf{Q}=\eps \mathbf{E},\qquad \mathbf{P}=s\mathbf{Q}.
\end{equation}
The continuous state is $Y_2=(\mathbf{D},\mathbf{H},\mathbf{Q},\mathbf{V})$, with
\begin{equation}\label{eq:jump2}
 \jump{\mathbf{D}}=\jump{\mathbf{H}}=\jump{\mathbf{Q}}=\jump{\mathbf{V}}=0,\qquad
 \mathbf{E}^+=\mathbf{E}^- -\frac{s^+-s^-}{\eps}\mathbf{Q}.
\end{equation}
Continuity of $\mathbf{P}$ must not be imposed: the continuous material coordinate in this model is $\mathbf{Q}$, while the constitutive reconstruction is $\mathbf{P}=s\mathbf{Q}$. A jump in $s$ therefore produces the physical jump $\mathbf{P}^+-\mathbf{P}^-=(s^+-s^-)\mathbf{Q}$; imposing continuity of $\mathbf{P}$ would contradict this constitutive law unless the jump happens to vanish. The energy is
\begin{equation}\label{eq:energy2}
 \Et_2=\frac12\int_\Omega\left(\frac{|\mathbf{D}-s\mathbf{Q}|^2}{\eps}
 +\muo|\mathbf{H}|^2+\frac{s}{\eps}(|\mathbf{V}|^2+b|\mathbf{Q}|^2)\right)dx.
\end{equation}

\subsubsection{Model 3: modulated current density}
This is the modulated-current-density mechanism of Dixon et al.~\cite{Dixon2025}; its integral constitutive law requires an independent memory variable rather than the algebraic identification $\mathbf{P}=s\mathbf{Q}$.
The integral constitutive law is localized by an independent memory state:
\begin{equation}\label{eq:m3}
 \partial_t\mathbf{P}=s\mathbf{V},\qquad \partial_t\mathbf{Q}=\mathbf{V},\qquad \partial_t\mathbf{V}+\gamma \mathbf{V}+b\mathbf{Q}=\eps \mathbf{E}.
\end{equation}
Here $\mathbf{V}=(\partial_t\mathbf{P})/s$, and $\mathbf{Q}(t)=\mathbf{Q}(\theta_0)+\int_{\theta_0}^t \mathbf{V}(r)dr$. The integral convention in \cite{Dixon2025} corresponds to $\mathbf{Q}(\theta_0)=0$ at the stated reference time. If computation starts later, the inherited memory must be supplied. 

The continuous state is $Y_3=(\mathbf{D},\mathbf{H},\mathbf{P},\mathbf{Q},\mathbf{V})$. All five components are continuous at events. Consequently, $\mathbf{E}$ is continuous and $\partial_t\mathbf{P}=s\mathbf{V}$ can jump. In general $\mathbf{P}\ne s\mathbf{Q}$. Use
\begin{equation}\label{eq:energy3}
 \Et_3=\frac12\int_\Omega\left(\eps|\mathbf{E}|^2+\muo|\mathbf{H}|^2
 +\frac{s}{\eps}(|\mathbf{V}|^2+b|\mathbf{Q}|^2)\right)dx.
\end{equation}
This energy does not control the independent $\mathbf{P}$ state. We account for that state separately in stability and error estimates.

\subsubsection{Model 4: modulated effective mass}
This is the modulated-effective-mass mechanism of Dixon et al.~\cite{Dixon2025}, represented here by the coefficient $s$ multiplying the polarization current and by the reciprocal factor $b/s$ in the restoring term.
With $\mathbf{V}=(\partial_t\mathbf{P})/s$, the equations become
\begin{equation}\label{eq:m4}
 \partial_t\mathbf{P}=s\mathbf{V},\qquad \partial_t\mathbf{V}+\gamma \mathbf{V}+\frac{b}{s}\mathbf{P}=\eps \mathbf{E}.
\end{equation}
The continuous state is $Y_4=(\mathbf{D},\mathbf{H},\mathbf{P},\mathbf{V})$, and all components are continuous at events. The energy is
\begin{equation}\label{eq:energy4}
 \Et_4=\frac12\int_\Omega\left(\eps|\mathbf{E}|^2+\muo|\mathbf{H}|^2
 +\frac{s}{\eps}|\mathbf{V}|^2+\frac{b}{\eps s}|\mathbf{P}|^2\right)dx.
\end{equation}
For $b=0$, the field and polarization dynamics of models 3 and 4 coincide, although the coercivity assumption used below no longer holds.

\begin{center}\small
\begin{tabular}{lll}\toprule
Model & Continuous state & State controlled by its physical energy\\\midrule
1 & $(\mathbf{D},\mathbf{H},\mathbf{P},\mathbf{J})$ & all components\\
2 & $(\mathbf{D},\mathbf{H},\mathbf{Q},\mathbf{V})$ & all components\\
3 & $(\mathbf{D},\mathbf{H},\mathbf{P},\mathbf{Q},\mathbf{V})$ & $(\mathbf{E},\mathbf{H},\mathbf{Q},\mathbf{V})$; $\mathbf{P}$ is controlled separately\\
4 & $(\mathbf{D},\mathbf{H},\mathbf{P},\mathbf{V})$ & all components\\\bottomrule
\end{tabular}
\end{center}
The table distinguishes the full continuous state from the variables directly controlled by the energy; only model 3 requires a separate estimate for its independent polarization.

Set $\Et_{\rm em}=\tfrac12\int(\eps|\mathbf{E}|^2+\muo|\mathbf{H}|^2)$. Under the PEC condition, integration by parts gives
\begin{equation}\label{eq:poynting}
 \partial_t\Et_{\rm em}=-(\mathbf{E},\partial_t\mathbf{P})_\Omega.
\end{equation}
Indeed, differentiating $\Et_{\rm em}$ and using $\partial_t\mathbf{D}=\curl\mathbf{H}$ and $\muo\partial_t\mathbf{H}=-\curl\mathbf{E}$ gives
\[
 \partial_t\Et_{\rm em}=(\mathbf{E},\curl\mathbf{H})_\Omega-(\mathbf{H},\curl\mathbf{E})_\Omega-(\mathbf{E},\partial_t\mathbf{P})_\Omega.
\]
The curl integration identity turns the first two terms into $-\int_{\partial\Omega}\mathbf{n}\cdot(\mathbf{E}\times\mathbf{H})\,dS$. This boundary flux vanishes for PEC data, since the tangential electric field is zero, and cancels on paired periodic faces. This yields \eqref{eq:poynting}.
The energies used here describe the chosen field--material equations. They do not include the energy of an external modulation device.

\begin{proposition}[Continuous balances]\label{prop:balance}
For sufficiently regular solutions on an event-free interval,
\begin{equation}\label{eq:balance}
 \Et_i'=-\mathcal L_i+\mathcal W_i,
\end{equation}
where
\begin{align}
 \mathcal L_1&=\int_\Omega\frac{\gamma}{\eps s}|\mathbf{J}|^2dx,
 &\mathcal W_1&=\frac1{2\eps}\int_\Omega\left[\partial_t(s^{-1})|\mathbf{J}|^2+\partial_t(b/s)|\mathbf{P}|^2\right]dx,\label{eq:w1}\\
 \mathcal L_2&=\int_\Omega\frac{\gamma s}{\eps}|\mathbf{V}|^2dx,
 &\mathcal W_2&=\frac1{2\eps}\int_\Omega\left[(\partial_t s)|\mathbf{V}|^2+\partial_t(sb)|\mathbf{Q}|^2\right]dx
 -\int_\Omega (\partial_t s)\mathbf{E}\cdot \mathbf{Q}\,dx,\label{eq:w2}\\
 \mathcal L_3&=\int_\Omega\frac{\gamma s}{\eps}|\mathbf{V}|^2dx,
 &\mathcal W_3&=\frac1{2\eps}\int_\Omega\left[(\partial_t s)|\mathbf{V}|^2+\partial_t(sb)|\mathbf{Q}|^2\right]dx,\label{eq:w3}\\
 \mathcal L_4&=\int_\Omega\frac{\gamma s}{\eps}|\mathbf{V}|^2dx,
 &\mathcal W_4&=\frac1{2\eps}\int_\Omega\left[(\partial_t s)|\mathbf{V}|^2+\partial_t(b/s)|\mathbf{P}|^2\right]dx.\label{eq:w4}
\end{align}
\end{proposition}
\begin{proof}
For model 1, differentiating the material energy and substituting \eqref{eq:m1} gives $(\mathbf{E},\mathbf{J})-\mathcal L_1+\mathcal W_1$. For model 2, its material derivative gives $(\mathbf{E},s\mathbf{V})-\mathcal L_2$ plus the differentiated weights, whereas \eqref{eq:poynting} equals $-(\mathbf{E},s\mathbf{V}+(\partial_t s)\mathbf{Q})$. For model 3 the electromagnetic contribution is $-(\mathbf{E},s\mathbf{V})$, with no $(\partial_t s)\mathbf{Q}$ term because $\mathbf{P}$ is independent. For model 4, the two oscillator coupling terms $-(b/\eps)\mathbf{P}\cdot \mathbf{V}$ and $(b/\eps)\mathbf{P}\cdot \mathbf{V}$ cancel. This proves all four identities.
\end{proof}

At an event, evaluate the one-sided energies using the common continuous state. In models 1, 3, and 4 this gives, respectively,
\begin{align}
 \jump{\Et_1}&=\frac1{2\eps}\int\left[\jump{s^{-1}}|\mathbf{J}|^2+\jump{b/s}|\mathbf{P}|^2\right]dx,\label{eq:ej1}\\
 \jump{\Et_3}&=\frac1{2\eps}\int\left[\jump s|\mathbf{V}|^2+\jump{sb}|\mathbf{Q}|^2\right]dx,\label{eq:ej3}\\
 \jump{\Et_4}&=\frac1{2\eps}\int\left[\jump s|\mathbf{V}|^2+\jump{b/s}|\mathbf{P}|^2\right]dx.\label{eq:ej4}
\end{align}
For model 2, including the electromagnetic jump is essential:
\begin{equation}\label{eq:ej2}
 \jump{\Et_2}=\frac1{2\eps}\int_\Omega\left[
 |\mathbf{D}-s^+\mathbf{Q}|^2-|\mathbf{D}-s^-\mathbf{Q}|^2+\jump s|\mathbf{V}|^2+\jump{sb}|\mathbf{Q}|^2\right]dx.
\end{equation}
For $i=1,2,4$, write $\Et_i(Y)=\tfrac12\int Y^TM_i(t)Y$. In the state order used below, the energy matrices (each scalar entry acts on the corresponding three-vector block) are
\[
 M_1=\frac1\eps\begin{pmatrix}
 1&0&-1&0\\0&\eps\muo&0&0\\-1&0&1+b/s&0\\0&0&0&1/s
 \end{pmatrix},\quad Y_1=(\mathbf{D},\mathbf{H},\mathbf{P},\mathbf{J}),
\]
\[
 M_2=\frac1\eps\begin{pmatrix}
 1&0&-s&0\\0&\eps\muo&0&0\\-s&0&s^2+sb&0\\0&0&0&s
 \end{pmatrix},\quad Y_2=(\mathbf{D},\mathbf{H},\mathbf{Q},\mathbf{V}),
\]
\[
 M_4=\frac1\eps\begin{pmatrix}
 1&0&-1&0\\0&\eps\muo&0&0\\-1&0&1+b/s&0\\0&0&0&s
 \end{pmatrix},\quad Y_4=(\mathbf{D},\mathbf{H},\mathbf{P},\mathbf{V}).
\]
These matrices reproduce \eqref{eq:energy1}, \eqref{eq:energy2}, and \eqref{eq:energy4}, respectively, and are positive definite under Assumption~\ref{ass:coeff}. Define
\begin{equation}\label{eq:cmod}
 C_i(t)=\mathop{\rm ess\,sup}_{x\in\Omega}
 \max\{0,\lambda_{\max}(M_i^{-1/2}(\partial_tM_i)M_i^{-1/2})\}.
\end{equation}
Here $M_i$ is the energy weight (symmetrizer), not a time-independent mass matrix of the evolution equations. Its entries depend on $s(t)$ and $b(t)$, so it generally varies even on an event-free interval. For a differentiable state on such an interval, the product rule gives
\[
 \partial_t\Et_i(Y)
 =\int_\Omega (\partial_tY)^TM_iY\,dx
  +\frac12\int_\Omega Y^T(\partial_tM_i)Y\,dx,
 \qquad
 \frac12\int_\Omega Y^T(\partial_tM_i)Y\,dx
 \le C_i(t)\Et_i(Y).
\]
The second integral is the modulation work in the energy coordinates used above. The same product rule applies to an error state: this term controls energy growth and is not an additional consistency forcing. If $s,b$ are constant between events, $\partial_tM_i=0$ there; coefficient jumps are treated separately by the energy ratios below, not by differentiating $M_i$ across a jump.
For model 3 use the reduced state $(\mathbf{E},\mathbf{H},\mathbf{Q},\mathbf{V})$ and its diagonal energy matrix. The differentiated weights in these coordinates are exactly \eqref{eq:w3}.
At event $\theta_j$, let
\begin{equation}\label{eq:rho}
 \rho_{i,j}=\mathop{\rm ess\,sup}_x
 \lambda_{\max}((M_i^-)^{-1/2}M_i^+(M_i^-)^{-1/2}).
\end{equation}
For model 3 the same definition uses the reduced state, which is continuous. Thus $\Et_i^+\le\rho_{i,j}\Et_i^-$. For instance,
\[
 \rho_{1,j}=\mathop{\rm ess\,sup}_x\max\left\{1,\frac{s^-}{s^+},\frac{b^+/s^+}{b^-/s^-}\right\},\quad
 \rho_{4,j}=\mathop{\rm ess\,sup}_x\max\left\{1,\frac{s^+}{s^-},\frac{b^+/s^+}{b^-/s^-}\right\}.
\]
Define the modulation amplification bound
\begin{equation}\label{eq:G}
 G_i(T)=\exp\left(\sum_j\int_{I_j}C_i(t)dt\right)
 \prod_{j=1}^{N_{\rm ev}}\max\{1,\rho_{i,j}\}.
\end{equation}
This bound is an energy estimate, not a sharp Floquet growth prediction.

\section{HDG formulation and semidiscrete stability}\label{sec:hdg}
We distinguish two polynomial/stabilization variants of one mixed Maxwell HDG formulation. This distinction is separate from the two time-domain formulations of Berardocco et al.~\cite{Berardocco2020}, which use different underlying first-order or diffusion-equation systems. Variant B shares the $k+1/k$ degree pattern and the $h^{-1}$-scaled stabilization used in superconvergent Maxwell HDG constructions such as Chen--Qiu--Shi--Solano~\cite{ChenQiuShiSolano2017}. It is nevertheless a different method: it discretizes a first-order time-dependent Maxwell--material system, contains no divergence multiplier, uses the full tangential electric trace, and is analyzed through a semidiscrete energy identity rather than a stationary curl--curl duality argument. For $k\ge0$ define
\[
 W_h^E=\{\mathbf{v}:\mathbf{v}|_K\in[\mathcal P_m(K)]^3\},\qquad
 W_h^H=\{\mathbf{w}:\mathbf{w}|_K\in[\mathcal P_k(K)]^3\}.
\]
Let $M_h^0$ contain single-valued tangential vector polynomials of degree $m$ on facets, zero on PEC facets. The choices are
\begin{center}
\begin{tabular}{lccc}
\toprule
Variant & Electric/material degree $m$ & Magnetic degree & Stabilization\\
\midrule
A (equal degree) & $k$ & $k$ & $\tau_K\simeq1$\\
B (enriched) & $k+1$ & $k$ & $\tau_K=\theta_K/h_K$\\
\bottomrule
\end{tabular}
\end{center}
Here positive upper and lower bounds on $\theta_K$ are independent of mesh size and time; $\theta_K$ has the units of admittance times length. Variant A has uniformly positive mesh-independent $\tau_K$. The chosen $\tau_K$ is fixed in time. Both variants use the full electric tangential trace, of degree $m$. Write $\mathbf{v}_T=\mathbf{v}-(\mathbf{v}\cdot \mathbf{n})\mathbf{n}$. Every material state and $\mathbf{D}_h$ are approximated in $W_h^E$. The sole independent hybrid variable is $\widehat{\mathbf{E}}_h\in M_h^0$. Variant B has more electric and facet unknowns and a different stabilization scale; its proven common rate is optimal for the magnetic degree, not an extra-order electric-field result.

Given the material variables, define $\mathbf{E}_h$ algebraically from $\mathbf{D}_h$: $\mathbf{E}_h=(\mathbf{D}_h-\mathbf{P}_h)/\eps$ for models 1, 3, and 4, and $\mathbf{E}_h=(\mathbf{D}_h-s\mathbf{Q}_h)/\eps$ for model 2. These belong to $W_h^E$ because $s$ is spatially constant on each element. Seek the element fields and trace such that
\begin{subequations}    
    \label{eq:hdg-eqs}
\begin{align}
 \ip{\partial_t\mathbf{D}_h}{\mathbf{v}}-\ip{\mathbf{H}_h}{\curl \mathbf{v}}-\bd{\mathbf{n}\times\widehat{\mathbf{H}}_h}{\mathbf{v}}&=0,\label{eq:hdg1}\\
 \ip{\muo \partial_t\mathbf{H}_h}{\mathbf{w}}+\ip{\mathbf{E}_h}{\curl \mathbf{w}}+\bd{\mathbf{n}\times\widehat{\mathbf{E}}_h}{\mathbf{w}}&=0,\label{eq:hdg2}\\
 \bd{\mathbf{n}\times\widehat{\mathbf{H}}_h}{\widehat{\mathbf{v}}}&=0\qquad(\widehat{\mathbf{v}}\in M_h^0),\label{eq:hdg3}
\end{align}
\end{subequations}
for $\mathbf{v}\in W_h^E$, $\mathbf{w}\in W_h^H$, with the outward-normal convention
\begin{equation}\label{eq:flux}
 \mathbf{n}\times\widehat{\mathbf{H}}_h=\mathbf{n}\times \mathbf{H}_h-\tau_K(\mathbf{E}_{h,T}-\widehat{\mathbf{E}}_h).
\end{equation}
The $h^{-1}$ scale in variant B is chosen for the convergence proof and is not the classical wave-impedance upwind value.

Material equations \eqref{eq:m1}--\eqref{eq:m4} are imposed weakly against all tests in $W_h^E$. Since their coefficients are elementwise constant in space, they also hold as polynomial identities on each element. 

\subsection{Temporal interfaces and initialization}
At an event, copy the model's continuous states exactly:
\begin{equation}\label{eq:transfer}
 Y_{i,h}(\theta_j^+)=Y_{i,h}(\theta_j^-).
\end{equation}
Then recompute $\mathbf{E}_h$ using the new coefficient and recompute the algebraic trace from \eqref{eq:hdg3}. The trace is not an independent differential state and need not be copied across events. This observation is especially important in model 2.

Use elementwise $L^2$ projections of the continuous initial states: the degree-$m$ electric projection for $\mathbf{D}$ and all material states, and the degree-$k$ projection for $\mathbf{H}$. In model 3, this includes the initial memory. Divergence-conforming initialization can be studied separately if discrete Gauss preservation is required.

\subsection{Solvability and local elimination}
For a fixed pair of element fields, \eqref{eq:hdg3} determines the trace uniquely. Eliminating the trace yields a finite-dimensional linear ODE for the continuous states, with piecewise bounded coefficients. At fixed material states, the algebraic transformation between $\mathbf{D}_h$ and $\mathbf{E}_h$ is invertible because $\eps>0$. Positivity of $s$ instead makes the scaled variables $\mathbf{P}/s$ and $(\partial_t\mathbf{P})/s$ well defined and supplies the positive material energy weights. The ODE therefore has a unique solution on each interval and is continued by \eqref{eq:transfer}.

In an implicit time step, the material variables are element-local and can be eliminated together with the element electromagnetic fields. The resulting global system involves only $\widehat{\mathbf{E}}_h$. Existence and uniqueness of solutions for implicit Euler and midpoint methods are obtained from a form of energy inequality.
%For models 1, 2, and 4, uniqueness of an implicit Euler step also follows from the energy identity below. For model 3, apply the reduced frozen energy to the homogeneous local backward Euler problem with zero incoming state and zero trace. Its nonnegative energy, increment, loss, and boundary terms imply $\mathbf{E}_h=\mathbf{H}_h=\mathbf{Q}_h=\mathbf{V}_h=0$. The remaining equation $\mathbf{P}_h/\Delta t=s\mathbf{V}_h$ gives $\mathbf{P}_h=0$, and $\mathbf{D}_h=\eps \mathbf{E}_h+\mathbf{P}_h=0$. The square finite-dimensional local operator is therefore injective and hence invertible for all four models. Testing the midpoint local problem with its stage state gives the corresponding uniqueness result.

\subsection{Semidiscrete balance and stability}
Set numerical dissipation
\begin{equation}\label{eq:diss}
 \Dh=\sum_K\tau_K\norm{\mathbf{E}_{h,T}-\widehat{\mathbf{E}}_h}_{\partial K}^2.
\end{equation}
\begin{theorem}[Semidiscrete energy balance]\label{thm:semi}
Under Assumption~\ref{ass:coeff}, the HDG solution satisfies, on every event-free interval,
\begin{equation}\label{eq:semibalance}
 \partial_t \Et_{i,h}+\Dh+\mathcal L_{i,h}=\mathcal W_{i,h},
\end{equation}
where the energies, loss, and modulation terms are exactly \eqref{eq:energy1}--\eqref{eq:energy4} and \eqref{eq:w1}--\eqref{eq:w4}, evaluated at discrete fields. The event identities \eqref{eq:ej1}--\eqref{eq:ej2} hold without additional projection loss. Consequently,
\begin{equation}\label{eq:semistab}
 \sup_{0\le t\le T}\Et_{i,h}(t)\le G_i(T)\Et_{i,h}(0).
\end{equation}
\end{theorem}
\begin{proof}
Choose $\mathbf{v}=\mathbf{E}_h$, $\mathbf{w}=\mathbf{H}_h$, and $\widehat{\mathbf{v}}=\widehat{\mathbf{E}}_h$ in \eqref{eq:hdg-eqs}. The element curl integration identity gives the combined boundary term
\[
 \bd{\mathbf{n}\times \mathbf{H}_h}{\mathbf{E}_h-\widehat{\mathbf{E}}_h}
 -\bd{\mathbf{n}\times\widehat{\mathbf{H}}_h}{\mathbf{E}_h}
 =\bd{\mathbf{n}\times \mathbf{H}_h-\mathbf{n}\times\widehat{\mathbf{H}}_h}{\mathbf{E}_h-\widehat{\mathbf{E}}_h}
 =\Dh,
\]
where the conservativity equation \eqref{eq:hdg3}, namely $\bd{\mathbf{n}\times\widehat{\mathbf{H}}_h}{\widehat{\mathbf{v}}}=0$ for all $\widehat{\mathbf{v}}\in M_h^0$, was used to subtract the trace from the second pairing. Hence
\[
 \partial_t\Et_{{\rm em},h}+\Dh=-(\partial_t\mathbf{P}_h,\mathbf{E}_h).
\]
The material cancellation is the same as in Proposition~\ref{prop:balance}. The element-wise coefficient assumption makes every required weighted test admissible. State transfer gives the event identities. Finally, $\mathcal W_{i,h}\le C_i\Et_{i,h}$ and $\mathcal L_{i,h},\Dh\ge0$ imply the bound by interval-wise integration and event multiplication.
\end{proof}

\begin{remark}[The additional state in model 3]
Theorem~\ref{thm:semi} controls $\mathbf{E}_h,\mathbf{H}_h,\mathbf{Q}_h,\mathbf{V}_h$ but not $\mathbf{P}_h$ alone. The independent equation gives
\[
 \norm{\mathbf{P}_h(t)}\le\norm{\mathbf{P}_h(0)}+s^*\int_0^t\norm{\mathbf{V}_h(r)}dr
 \le\norm{\mathbf{P}_h(0)}+C T\sqrt{G_3(T)\Et_{3,h}(0)}.
\]
Then $\mathbf{D}_h=\eps \mathbf{E}_h+\mathbf{P}_h$ is controlled. This is preferable to describing the physical energy as coercive on all five states.
\end{remark}

\section{Spatial error analysis}\label{sec:error}
The energy argument applies to both variants in Section~\ref{sec:hdg}. The enriched variant attains an $h^{k+1}$ common field/state rate using block $L^2$ projections and balancing the two boundary residuals. This is optimal for the degree-$k$ magnetic field; it is not an $h^{k+2}$ result for the degree-$(k+1)$ electric field. The equal-degree variant A remains a distinct, cheaper variant with the baseline estimate stated below.

\subsection{Projection identities and the exact error equations}
Let $m=k$ for variant A and $m=k+1$ for variant B. Write $\Pi_E$ for the elementwise vector $L^2$ projection onto $[\mathcal P_m]^3$, $\Pi_H$ for the projection onto $[\mathcal P_k]^3$, and $\Pi_M$ for the tangential facet projection of degree $m$. Project every material state and $\mathbf{D}$ with $\Pi_E$, and $\mathbf{H}$ with $\Pi_H$. These operators are independent of time.

\begin{assumption}[Solution regularity and transmission]\label{ass:solution}
On every event-free interval $I_j$, each continuous state belongs to $W^{1,1}(I_j;L^2(\Omega))$ and has one-sided endpoint values. The states satisfy the model equations and the exact temporal transmission rules. Electric and magnetic tangential traces are single-valued across interior and paired periodic facets; the electric trace vanishes on PEC facets, also at one-sided event limits. For variant B, $\mathbf{H}$ is uniformly in broken $\mathbf{H}^{k+1}$ and $\mathbf{E}$ and every independent material state are uniformly in broken $\mathbf{H}^{k+2}$ on each closed slab: the quantities $\mathcal A_i,\mathcal B$ defined below are finite, uniformly over refinement of a fixed mesh-resolved material partition. For variant A, all these fields instead have uniformly bounded broken $\mathbf{H}^{k+1}$ norms. Additional temporal derivatives required by a time-discrete estimate are stated in that estimate.
\end{assumption}
These assumptions imply that the time-independent block projections commute with weak time differentiation. For each variable $\xi$, use $e_\xi^h=\Pi_\xi\xi-\xi_h$ for the discrete error and $e_\xi^I=\Pi_\xi\xi-\xi$ for the projection error. In particular,
\begin{align*}
 \mathbf{e}_E^h&=\Pi_E \mathbf{E}-\mathbf{E}_h,& \mathbf{e}_H^h&=\Pi_H \mathbf{H}-\mathbf{H}_h,&\widehat{\mathbf{e}}_{E_T}^h&=\Pi_M \mathbf{E}_T-\widehat{\mathbf{E}}_h,\\
 \mathbf{e}_E^I&=\Pi_E \mathbf{E}-\mathbf{E},&\mathbf{e}_H^I&=\Pi_H \mathbf{H}-\mathbf{H},&\mathbf{e}_{E_T}^I&=\Pi_M \mathbf{E}_T-\mathbf{E}_T.
\end{align*}
For each electric/material state $Z$, set $\mathbf{e}_Z^h=\Pi_E Z-Z_h$; let $\mathbf{e}_i^h$ collect the discrete errors for model $i$. Set $\mathbf{e}_{E,T}^h=(\mathbf{e}_E^h)_T$ and define the magnetic flux error through the last identity in \eqref{eq:errflux}. The constitutive relations commute exactly, including
\[
 \mathbf{e}_D^h=\eps \mathbf{e}_E^h+\mathbf{e}_P^h\quad(i=1,3,4),\qquad \mathbf{e}_D^h=\eps \mathbf{e}_E^h+s \mathbf{e}_Q^h\quad(i=2).
\]
On each element define
\begin{equation}\label{eq:R}
 R_K=\mathbf{n}\times\mathbf{e}_H^I-\tau_K(\mathbf{e}_{E,T}^I-\mathbf{e}_{E_T}^I),\qquad
 F_h^2(t)=\sum_K\tau_K^{-1}\norm{R_K}_{\partial K}^2.
\end{equation}
The remainder is an $L^2$ tangential field and need not be polynomial.

\begin{lemma}[Exact projection compatibility]\label{lem:compat}
For either variant, the projected material equations are exact. The projected electromagnetic error satisfies
\begin{align}
 \ip{\partial_t\mathbf{e}_D^h}{\mathbf{v}}-\ip{\mathbf{e}_H^h}{\curl \mathbf{v}}-\bd{\mathbf{n}\times\widehat{\mathbf{e}}_H^h}{\mathbf{v}}
   &=-\sum_K\langle R_K,\mathbf{v}_T\rangle_{\partial K},\label{eq:erra}\\
 \ip{\muo \partial_t\mathbf{e}_H^h}{\mathbf{w}}+\ip{\mathbf{e}_E^h}{\curl \mathbf{w}}+\bd{\mathbf{n}\times\widehat{\mathbf{e}}_{E_T}^h}{\mathbf{w}}&=0,\label{eq:errf}\\
 \bd{\mathbf{n}\times\widehat{\mathbf{e}}_H^h}{\widehat{\mathbf{v}}}
   &=\sum_K\langle R_K,\widehat{\mathbf{v}}\rangle_{\partial K},\label{eq:errc}\\
 \mathbf{n}\times\widehat{\mathbf{e}}_H^h&=\mathbf{n}\times \mathbf{e}_H^h-\tau_K(\mathbf{e}_{E,T}^h-\widehat{\mathbf{e}}_{E_T}^h).\label{eq:errflux}
\end{align}
Here $\mathbf{v}\in W_h^E$, $\mathbf{w}\in W_h^H$, and $\widehat{\mathbf{v}}\in M_h^0$.
\end{lemma}
\begin{proof}
Each material equation is a linear combination of electric/material states with scalar coefficients constant on $K$. Applying $\Pi_E$ therefore gives exactly the discrete material equation. No common projection of $\mathbf{E}$ and $\mathbf{H}$ is needed for this step.

Mass projection defects vanish against the corresponding volume tests. Furthermore,
\[
 \curl W_h^E\subset W_h^H,\qquad \curl W_h^H\subset W_h^E,
\]
for $m=k$ or $m=k+1$. Thus $(\mathbf{e}_H^I,\curl \mathbf{v})_K=(\mathbf{e}_E^I,\curl \mathbf{w})_K=0$. The tangential trace of $\mathbf{w}$ has degree at most $k\le m$, so $\langle \mathbf{n}\times\mathbf{e}_{E_T}^I,\mathbf{w}\rangle_{\partial K}=0$ by facet projection orthogonality. The projected magnetic numerical flux differs from $\mathbf{n}\times \mathbf{H}$ by exactly $R_K$. Subtracting the discrete equations from these projected equations proves \eqref{eq:erra}--\eqref{eq:errflux}. Interior exact magnetic fluxes cancel against single-valued hybrid tests.
\end{proof}

\begin{lemma}[Projected-error energy identity]\label{lem:errorenergy}
Inside each event-free interval,
\begin{equation}\label{eq:errorbalance}
 \Et_i(\mathbf{e}_i^h)' +\Dh(\mathbf{e}_i^h)+\mathcal L_i(\mathbf{e}_i^h)
 =\mathcal W_i(\mathbf{e}_i^h)-\sum_K\langle R_K,\mathbf{e}_{E,T}^h-\widehat{\mathbf{e}}_{E_T}^h\rangle_{\partial K},
\end{equation}
where $\Dh(\mathbf{e}_i^h)=\sum_K\tau_K\norm{\mathbf{e}_{E,T}^h-\widehat{\mathbf{e}}_{E_T}^h}_{\partial K}^2$. Consequently,
\begin{equation}\label{eq:errineq}
 \Et_i(\mathbf{e}_i^h)'+\tfrac12\Dh(\mathbf{e}_i^h)+\mathcal L_i(\mathbf{e}_i^h)
 \le C_i(t)\Et_i(\mathbf{e}_i^h)+\tfrac12 F_h^2(t).
\end{equation}
\end{lemma}
\begin{proof}
Test \eqref{eq:erra} with $\mathbf{e}_E^h$ and \eqref{eq:errf} with $\mathbf{e}_H^h$. In the boundary cancellation the nonzero conservativity residual \eqref{eq:errc}, tested with $\widehat{\mathbf{e}}_{E_T}^h$, contributes $+\langle R,\widehat{\mathbf{e}}_{E_T}^h\rangle$ to the right-hand side. Thus the boundary forcing is precisely $-\langle R,\mathbf{e}_{E,T}^h-\widehat{\mathbf{e}}_{E_T}^h\rangle$. Material energy cancellation is exact by Lemma~\ref{lem:compat}. Weighted Cauchy--Schwarz bounds this forcing by $F_h\Dh(\mathbf{e}_i^h)^{1/2}$, and Young's inequality gives \eqref{eq:errineq}.
\end{proof}

\paragraph{Effect of one electric degree and the stabilization scale}
Assume the mesh is shape regular. Quasi-uniformity is unnecessary for the elementwise estimates below. For a degree-$\ell$ volume $L^2$ projection,
\begin{equation}\label{eq:traceapprox}
 \norm{\Pi_q u-u}_{\partial K}
 \le C h_K^{\ell+1/2}|u|_{H^{\ell+1}(K)}.
\end{equation}
This standard polynomial trace approximation estimate follows from the trace theorem, volume approximation in $L^2$ and $\mathbf{H}^1$, and polynomial inverse inequalities; see, for example, Di Pietro and Ern~\cite[Chapter 1]{DiPietroErn2012}. Facet best approximation also gives
\[
 \norm{\mathbf{e}_{E_T}^I}_{\partial K}\le\norm{\mathbf{e}_{E,T}^I}_{\partial K}.
\]
Hence
\begin{equation}\label{eq:resgeneric}
 F_h^2\le C\sum_K\left(
 \tau_K^{-1}h_K^{2k+1}|\mathbf{H}|_{H^{k+1}(K)}^2+
 \tau_K h_K^{2m+1}|\mathbf{E}|_{H^{m+1}(K)}^2\right).
\end{equation}
The constant depends on polynomial degree and shape regularity, not on $h_K$.

\begin{proposition}[Residual rates]\label{prop:resrates}
Variant A ($m=k$, $\tau_K\simeq1$) satisfies
\begin{equation}\label{eq:resA}
 F_h^2\le C\sum_K h_K^{2k+1}
 \left(|\mathbf{H}|_{H^{k+1}(K)}^2+|\mathbf{E}|_{H^{k+1}(K)}^2\right).
\end{equation}
Variant B ($m=k+1$, $\tau_K\simeq h_K^{-1}$) satisfies
\begin{equation}\label{eq:resB}
 F_h^2\le C\sum_K h_K^{2k+2}
 \left(|\mathbf{H}|_{H^{k+1}(K)}^2+|\mathbf{E}|_{H^{k+2}(K)}^2\right).
\end{equation}
\end{proposition}
\begin{proof}
Insert each degree and stabilization choice into \eqref{eq:resgeneric}. In variant B the magnetic contribution gains a factor $h_K$ from $\tau_K^{-1}$, while the electric contribution loses a factor $h_K$ but starts with the higher-degree trace estimate. Both powers become $h_K^{2k+2}$.
\end{proof}
Increasing $\tau$ alone in the equal-degree method would instead give an electric contribution of order $h^{2k}$. Likewise, merely assuming the remainder cancels is unjustified. For example, on a tetrahedron with $k=0$ a pair of constant three-dimensional fields has six unknowns, whereas a full tangential constant trace on four facets has eight independent moments. Arbitrary full-facet flux interpolation cannot be defined from those unknowns alone.

We retain the full degree-$m$ electric trace in variant B. If it were reduced to degree $k$ with only projected-jump stabilization, the high-degree part of $\mathbf{e}_{E,T}^h$ would not be controlled by $\Dh(\mathbf{e}_i^h)$, while $\mathbf{n}\times\mathbf{e}_H^I$ generally has nonzero pairing with that part. A separate generalized projection argument would then be necessary. The present proof neither assumes such cancellation nor claims the reduced-trace result. For context on space-dependent Maxwell projection constructions, see \cite{DuSayas2020,CockburnDuSanchez2023}.

\subsection{Events and integrated energy estimate}
For $i=1,2,4$, let $\Pi_iY_i$ denote the block projection using $\Pi_E$ for electric/material states and $\Pi_H$ for $\mathbf{H}$. Because the mesh and projection are fixed,
\begin{equation}\label{eq:errtransfer}
 \mathbf{e}_i^h(\theta_j^+)=\mathbf{e}_i^h(\theta_j^-).
\end{equation}
Therefore $\Et_i(\mathbf{e}_i^h(\theta_j^+))\le\rho_{i,j}\Et_i(\mathbf{e}_i^h(\theta_j^-))$ with exactly the physical energy-weight factor \eqref{eq:rho}. In model 2 the electric error jumps according to
\[
 \mathbf{e}_E^{h,+}=\mathbf{e}_E^{h,-} -\frac{s^+-s^-}{\eps}\mathbf{e}_Q^h,
\]
but the error in $(\mathbf{D},\mathbf{H},\mathbf{Q},\mathbf{V})$ remains continuous. Model 3 uses the continuous reduced error $(\mathbf{e}_E^h,\mathbf{e}_H^h,\mathbf{e}_Q^h,\mathbf{e}_V^h)$.

\begin{lemma}[Energy estimate including all events]\label{lem:integrated}
For projected initial continuous states,
\begin{equation}\label{eq:integrated}
 \sup_{0\le t\le T}\Et_i(\mathbf{e}_i^h(t))+\int_0^T\Dh(\mathbf{e}_i^h(t))dt
 \le C G_i(T)\int_0^T F_h^2(t)dt.
\end{equation}
An initial projected error adds $C G_i(T)\Et_i(\mathbf{e}_i^h(0))$.
\end{lemma}
\begin{proof}
On smooth intervals multiply \eqref{eq:errineq} by
\[
 w_i(t)=\exp\!\left(-\int_0^t C_i(r)dr\right)
 \prod_{\theta_j\le t}\max\{1,\rho_{i,j}\}^{-1}.
\]
The weighted energy has a nonpositive jump at each event. Integrating gives
\[
 w_i(t)\Et_i(\mathbf{e}_i^h(t))+\tfrac12\int_0^t w_i\Dh(\mathbf{e}_i^h)dr
 \le \Et_i(\mathbf{e}_i^h(0))+\tfrac12\int_0^t w_i F_h^2dr.
\]
Since $G_i(T)^{-1}\le w_i\le1$, the stated estimate follows (with an absolute constant accounting for the sum of the supremum and dissipation).
\end{proof}

\subsection{Spatial convergence theorem}
Set $h=\max_K h_K$. Let $Z_i=(\mathbf{E},\mathbf{H},\text{all independent material states})$; replacing $\mathbf{E}$ by $\mathbf{D}$ gives an equivalent full-state norm for models 1, 2, and 4. Define
\begin{align*}
 \mathcal A_i&=\sup_{t\in\cup_j\overline I_j}\left(
 \norm H_{H^{k+1}(\mathcal T_h)}+
 \sum_{Z\ne \mathbf{H}}\norm Z_{H^{k+2}(\mathcal T_h)}\right),\\
 \mathcal B^2&=\sum_j\int_{I_j}\left(
 |\mathbf{H}|_{H^{k+1}(\mathcal T_h)}^2+|\mathbf{E}|_{H^{k+2}(\mathcal T_h)}^2\right)dt.
\end{align*}
One-sided values are included in the suprema. Broken Sobolev norms sum over elements, and constants are understood for a fixed mesh-resolved material partition under refinement.

\begin{theorem}[Common $h^{k+1}$ rate for variant B]\label{thm:error}
Assume Assumptions~\ref{ass:coeff} and \ref{ass:solution}, and exact state transfer. Initialize every continuous state by its block $L^2$ projection. For $i=1,2,4$,
\begin{equation}\label{eq:spatialerror}
 \sup_{0\le t\le T}\norm{Z_i(t)-Z_{i,h}(t)}_{L^2(\Omega)}
 \le C h^{k+1}\left(\mathcal A_i+G_i(T)^{1/2}\mathcal B\right),
\end{equation}
and
\begin{equation}\label{eq:diss-error}
 \left(\int_0^T\sum_K\tau_K
 \norm{\mathbf{e}_{E,T}^h-\widehat{\mathbf{e}}_{E_T}^h}_{\partial K}^2dt\right)^{1/2}
 \le C G_i(T)^{1/2}h^{k+1}\mathcal B.
\end{equation}
For model 3, \eqref{eq:spatialerror} first holds for $(\mathbf{E},\mathbf{H},\mathbf{Q},\mathbf{V})$. Adding $\mathbf{P},\mathbf{D}$ preserves the rate, with a constant also depending on $T,s^*$.
\end{theorem}
\begin{proof}
Combine \eqref{eq:resB} and Lemma~\ref{lem:integrated}. The physical quadratic energy is uniformly equivalent to the indicated $L^2$ state norm for models 1, 2, and 4. Add the projection approximation: it is $O(h^{k+1})$ for $\mathbf{H}$ and $O(h^{k+2})$ for the degree-$(k+1)$ electric/material states. The energy error is $O(h^{k+1})$, so it limits the common rate.

For model 3 use the reduced physical energy. Its independent projected polarization satisfies
\[
 \partial_t\mathbf{e}_P^h=s \mathbf{e}_V^h,\qquad \mathbf{e}_P^h(0)=0,
\]
including continuity at all events. Thus $\sup_t\norm{\mathbf{e}_P^h(t)}\le s^*T\sup_t\norm{\mathbf{e}_V^h(t)}$. Add the projection error for $\mathbf{P}$ and use $\mathbf{D}=\eps \mathbf{E}+\mathbf{P}$. This proves both the full-state rate and the dissipation bound.
\end{proof}

\begin{corollary}[The equal-degree variant A]\label{cor:baseline}
Variant A, with projected initial data and the corresponding broken $\mathbf{H}^{k+1}$ regularity, satisfies the analogous common-state estimate of order $h^{k+1/2}$. It follows from \eqref{eq:resA} and Lemma~\ref{lem:integrated}. Theorem~\ref{thm:error} does not upgrade variant A to order $h^{k+1}$.
\end{corollary}

\begin{remark}[Regularity at heterogeneous switching events]\label{rem:compatible}
The regularity and spatial transmission assumptions are hypotheses about the exact solution, not consequences of bounded piecewise smooth coefficients. In particular, a spatially heterogeneous model-2 event can produce an electric jump whose tangential values are incompatible across a material interface. Here $\Delta s:=s^+-s^-$ denotes the change in the coupling coefficient across the event. Indeed, since $\mathbf{E}^+-\mathbf{E}^-=-(\Delta s)\mathbf{Q}/\eps$, preservation of a continuous pre-event tangential electric trace requires
\[
 \llbracket(\Delta s)\mathbf{Q}\rrbracket_T=0
 \quad\hbox{on internal material interfaces},\qquad
 \mathbf{n}\times((\Delta s)\mathbf{Q})=0\quad\hbox{on PEC boundaries}.
\]
Here $\llbracket\cdot\rrbracket_T$ denotes a spatial tangential jump, distinct from the temporal jump $\jump{\cdot}$. These are necessary trace compatibility conditions, not sufficient conditions for higher broken Sobolev regularity. Incompatible data may fail Assumption~\ref{ass:solution}. The theorem applies to compatible switching histories; lower regularity requires appropriately reduced approximation estimates. None of the constants is claimed uniform in the zero-coupling or zero-resonance limit.
\end{remark}

\section{Event-aligned time discretization: stability and fully discrete error}\label{sec:time}
For models 1, 2, and 4, let $Y_h$ denote the vector of all continuous-state coefficients after algebraic trace elimination. The semidiscrete system has the form
\begin{equation}\label{eq:ode}
 \partial_tY_h=A_h(t)Y_h.
\end{equation}
The discrete energy is $\Et_{h}(Y)=\tfrac12Y^TM_h(t)Y$, with a symmetric positive definite matrix. For fixed coefficients, the semidiscrete identity gives
\begin{equation}\label{eq:frozen}
 Y^TM_h(t)A_h(t)Y=-\Dh(Y;t)-\mathcal L_h(Y;t).
\end{equation}
The trace in these quadratic forms is the algebraic trace at time $t$.

\paragraph{Construction of an event-aligned grid.}
Let $0=\theta_0<\theta_1<\cdots<\theta_{N_{\rm ev}}<\theta_{N_{\rm ev}+1}=T$ be the prescribed
physical event times.  An event-aligned partition is any ordered set
$\mathcal G=\{0=\xi_0<\xi_1<\cdots<\xi_N=T\}$ that contains every $\theta_j$.
Thus no open time step contains an event.  Write
$\Delta t_n=\xi_n-\xi_{n-1}$ and
$\Delta t_{\max}=\max_n\Delta t_n$.

We use two concrete constructions.  Given a target maximum step $\Delta t>0$, the
backward Euler computations redistribute steps separately on every smooth
coefficient slab.  For $L_j=\theta_{j+1}-\theta_j$, set
\begin{equation}\label{eq:slabgrid}
 N_j=\left\lceil\frac{L_j}{\Delta t}\right\rceil,
 \qquad \Delta t_j=\frac{L_j}{N_j},
 \qquad \mathcal G_j=\{\theta_j+\ell \Delta t_j:\ell=0,\ldots,N_j\}.
\end{equation}
The union of the $\mathcal G_j$ is event aligned, has
$0<\Delta t_j\le \Delta t$, and is uniform within each slab.  In the fitted-midpoint timing
experiments we instead preserve a uniform base grid and insert the events:
\begin{equation}\label{eq:insertgrid}
 \mathcal G^{\rm ins}_{\Delta t}
 =\operatorname{sort}\left(
   \{\ell T/N:\ell=0,\ldots,N\}\cup\{\theta_1,\ldots,\theta_{N_{\rm ev}}\}
  \right),\qquad \Delta t=T/N,
\end{equation}
with duplicate nodes removed.  This construction changes only base steps that
contain an event and still has maximum step at most $\Delta t$.  Neither construction
uses a local-error estimator; the event times are input data.  The estimates
below apply to any event-aligned partition with positive steps, rather than to
only these two constructions.

At an event we distinguish the coincident one-sided states $Y_h^{j,-}$ and
$Y_h^{j,+}$.  The step arriving at $\theta_j$ uses the left coefficient limit, the
prescribed transmission map is then applied at zero elapsed time, and the next
step uses the right limit.  For the present continuous-state coordinates this
map copies the state, $Y_h^{j,+}=Y_h^{j,-}$, while the constitutive field and
hybrid trace are recomputed from the right coefficients.  In particular,
model 2 can have $\mathbf{E}_h^{j,+}\ne\mathbf{E}_h^{j,-}$ although its
continuous state is unchanged.

With $t_n=\xi_n$, both methods advance the continuous state on an event-free step:
\begin{align}
 \frac{Y_h^n-Y_h^{n-1}}{\Delta t_n}&=A_h(t_n^-)Y_h^n
 &&\text{(backward Euler)},\label{eq:be}\\
 \frac{Y_h^n-Y_h^{n-1}}{\Delta t_n}&=A_h(t_{n-1/2})\overline Y_h^n
 &&\text{(implicit midpoint)},\label{eq:midpoint}\\
 \overline Y_h^n&=\tfrac12(Y_h^n+Y_h^{n-1}).\nonumber
\end{align}
The minus sign at an event node selects the left coefficient for the arriving backward Euler step. After copying the state at that node, the next step uses the right coefficient. Midpoint samples strictly inside its slab. Material fields and the hybrid trace are reconstructed at the evaluation time. For model 3, use the closed reduced system for $(\mathbf{E}_h,\mathbf{H}_h,\mathbf{Q}_h,\mathbf{V}_h)$ and advance the independent $\mathbf{P}_h$ equation with the same time rule.

\begin{proposition}[Discrete energy stability of both methods]\label{prop:be}
For models 1, 2, and 4, let $t_n^*$ be $t_n^-$ for backward Euler and $t_{n-1/2}$ for midpoint. To distinguish the step energy metric from the model matrices $M_i$, set $\mathsf G_n^*=M_h(t_n^*)$ and $\Et_n^*(Z)=\tfrac12\norm Z_{\mathsf G_n^*}^2$. With $Z_n=Y_h^n$ for backward Euler and $Z_n=\overline Y_h^n$ for midpoint, both methods satisfy
\begin{equation}\label{eq:commonenergy}
 \Et_n^*(Y_h^n)-\Et_n^*(Y_h^{n-1})
 +\sigma\norm{Y_h^n-Y_h^{n-1}}_{\mathsf G_n^*}^2
 +\Delta t_n\bigl(\Dh(Z_n;t_n^*)+\mathcal L_h(Z_n;t_n^*)\bigr)=0,
\end{equation}
where $\sigma=1/2$ for backward Euler and $\sigma=0$ for midpoint. In particular, for either method the step is uniquely solvable and contractive in its frozen energy for every $\Delta t_n>0$. Across steps and events,
\begin{equation}\label{eq:discretestability}
 \max_n\Et_n^*(Y_h^n)\le G_i(T)\Et_1^*(Y_h^0)
\end{equation}
with one-sided interpretation at event nodes; no CFL condition is required. For model 3 the same assertions hold in the reduced energy, and its independent polarization is controlled by summing its update.
\end{proposition}
\begin{proof}
Pair each discrete equation with $Z_n$ in the frozen energy inner product. The semidiscrete cancellation \eqref{eq:frozen} gives dissipation and material loss. The identity
\[
 (Y_h^n-Y_h^{n-1},Z_n)_{\mathsf G_n^*}
 =\Et_n^*(Y_h^n)-\Et_n^*(Y_h^{n-1})
 +\sigma\norm{Y_h^n-Y_h^{n-1}}_{\mathsf G_n^*}^2
\]
proves \eqref{eq:commonenergy}. Applied to the homogeneous step with zero incoming state, it makes the outgoing state zero. The finite-dimensional square system is therefore invertible. To pass between frozen metrics, compare the unchanged state in successive $\mathsf G_n^*$; on a smooth slab the comparison is bounded by $\exp(\int C_i)$, and at an event by $\rho_{i,j}$. The smooth comparison intervals have disjoint interiors within each slab, so their total exponent is bounded by the slab integral in \eqref{eq:G}; a comparison crossing an event is split into two smooth pieces and its event factor. State transfer itself is the identity. This proves \eqref{eq:discretestability}. For model 3, subtract the polarization equation from the displacement equation to obtain a closed reduced system, apply the same calculation to its coercive energy, and then sum $\mathbf{P}_h^n-\mathbf{P}_h^{n-1}=\Delta t_n s(t_n^*)\mathbf{V}_h^n$ for backward Euler or $\Delta t_n s(t_n^*)\overline{\mathbf{V}}_h^n$ for midpoint.
\end{proof}

The extra increment term in backward Euler represents numerical damping; midpoint has no such term. For time-varying coefficients, \eqref{eq:commonenergy} compares states in a frozen metric and does not claim exact integration of the physical modulation work.
Unconditional stability here means solvability and the stated energy bounds without a CFL restriction. The finite-time amplification factor $G_i(T)$ can grow with $T$: neither monotone physical energy nor infinite-time exponential decay follows. The latter requires additional assumptions, as illustrated by the conditional decay analysis of Renner and Yousept~\cite{RennerYousept2026}.

\subsection{Backward Euler: fully discrete error}\label{sec:fullerror}
We compare the fully discrete solution directly with the block projection of the exact continuous solution. This avoids any assumption of mesh-uniform time derivatives of the semidiscrete solution.

Use a duplicate one-sided state at every event if necessary: the step arriving at $\theta_j$ uses left coefficient values, then the same continuous state is transferred and its right-side energy is evaluated. For an event-free step define
\begin{equation}\label{eq:exactdefect}
 d_{i,n}=\frac{Y_i(t_n)-Y_i(t_{n-1})}{\Delta t_n}-\partial_tY_i(t_n)
 =-\frac1{\Delta t_n}\int_{t_{n-1}}^{t_n}(r-t_{n-1})\partial_t^2Y_i(r)dr.
\end{equation}
Both the exact equation and the backward Euler equation use the exact coefficients at the same one-sided endpoint $t_n$. Their subtraction therefore introduces no separate coefficient-approximation residual: the effect of coefficient variation on time consistency is already included in $\partial_t^2Y_i$ in \eqref{eq:exactdefect}. Elementwise constant coefficients on the fixed material-resolving mesh also commute with the corresponding block projections. Approximating the coefficients themselves would introduce further residuals and is outside this estimate. The variation of $M_i$ instead enters the energy comparison through \eqref{eq:cmod} and the event ratios.

Let $g_{i,n}=\Pi_i d_{i,n}$. The matrices defining $\norm\cdot_{\mathsf G_i^n}$ include the volume integration; this norm is the physical continuous-state energy norm, not a coefficient Euclidean norm. Since coefficients are spatially constant on elements, $\Pi_i$ is contractive in each frozen physical norm: the electric/material blocks share the same projection and the magnetic block has no cross terms with them. Uniform coefficient bounds therefore imply
\begin{equation}\label{eq:defectbound}
 \norm{g_{i,n}}_{\mathsf G_i^n}\le C\Delta t_n
 \mathop{\rm ess\,sup}_{r\in(t_{n-1},t_n)}\norm{\partial_t^2Y_i(r)}_{L^2(\Omega)}.
\end{equation}

\begin{theorem}[Fully discrete common-state convergence]\label{thm:full}
Use variant B and projected initial continuous states. Assume the spatial hypotheses of Theorem~\ref{thm:error}, and, for $i=1,2,4$, piecewise $W^{2,\infty}(I_j;L^2)$ regularity of the exact continuous state $Y_i$. Set
\[
 \mathcal T_i=\max_j\norm{\partial_t^2Y_i}_{L^\infty(I_j;L^2)},\qquad
 \mathcal S=\sup_{t\in\cup_j\overline I_j}
 \left(|\mathbf{H}(t)|_{H^{k+1}(\mathcal T_h)}+|\mathbf{E}(t)|_{H^{k+2}(\mathcal T_h)}\right).
\]
For every event-aligned partition with positive steps, without a restriction on $\Delta t_{\max}$,
\begin{equation}\label{eq:fullrate}
 \max_n\norm{Z_i(t_n)-Z_{i,h}^n}_{L^2(\Omega)}
 \le C\left[h^{k+1}(\mathcal A_i+\mathcal S)
                 +\Delta t_{\max}\mathcal T_i\right].
\end{equation}
Here $C$ can depend on $T$, coefficient and shape-regularity bounds, and $G_i(T)$, but is independent of $h$, the number of steps, and $\Delta t_{\max}$.
\end{theorem}
\begin{proof}
Set $e_i^{h,n}=\Pi_iY_i(t_n)-Y_{i,h}^n$ and $\widehat{\mathbf{e}}_{E_T}^{h,n}=\Pi_M \mathbf{E}_T(t_n)-\widehat{\mathbf{E}}_h^n$. Projection commutes with backward differences because the spatial operators are fixed. Lemma~\ref{lem:compat} therefore gives the same spatial residual $R^n$, plus the volume time defects $g_{i,n}$ in all differential state equations. No discrete inverse estimate enters this argument.

Testing these equations in the frozen physical energy and using the polarization identity gives
\begin{align}\label{eq:fullerridentity}
 \Et_i^n(e_i^{h,n})&+\tfrac12\norm{e_i^{h,n}-e_i^{h,n-1}}_{\mathsf G_i^n}^2
 +\Delta t_n(\Dh(e^n)+\mathcal L_i^n(e^n))\\
 &=\Et_i^n(e_i^{h,n-1})
 -\Delta t_n\sum_K\langle R_K^n,\mathbf{e}_{E,T}^{h,n}-\widehat{\mathbf{e}}_{E_T}^{h,n}\rangle_{\partial K}
 +\Delta t_n(e_i^{h,n},g_{i,n})_{\mathsf G_i^n}.\nonumber
\end{align}
The facet term is bounded by $\tfrac12\Delta t_n(\Dh(e^n)+(F_h^n)^2)$. Keep the time-defect term linear in the error norm:
\[
 |(e_i^{h,n},g_{i,n})_{\mathsf G_i^n}|
 \le\norm{e_i^{h,n}}_{\mathsf G_i^n}\norm{g_{i,n}}_{\mathsf G_i^n}.
\]
For each step, let $r_n$ compare the energy of the unchanged preceding error in the current frozen metric with its preceding endpoint energy. Include a zero-length comparison step at each event, with $r_n=\rho_{i,j}$ and unchanged continuous error. Set $R_0=1$, $R_n=\prod_{\ell=1}^n\max\{1,r_\ell\}$, and $w_n=R_n^{-1}$. The smooth-slab comparison and event identities give $1\le R_n\le G_i(T)$. Also $w_n\Et_i^n(e_i^{h,n-1})\le w_{n-1}\Et_i^{n-1}(e_i^{h,n-1})$, with the analogous inequality at an event. Dropping nonnegative dissipation and increment terms, summing through index $m$, and multiplying by two gives
\[
 w_m\norm{e_i^m}_{M_i^m}^2
 \le\norm{e_i^0}_{M_i^0}^2
   +\sum_{n=1}^m\Delta t_n w_n(F_h^n)^2
   +2\sum_{n=1}^m\Delta t_n w_n
       \norm{e_i^{h,n}}_{\mathsf G_i^n}\norm{g_{i,n}}_{\mathsf G_i^n}.
\]
Let $U_m=\max_{0\le n\le m}\sqrt{w_n}\norm{e_i^{h,n}}_{\mathsf G_i^n}$, and apply the preceding inequality at an index realizing this maximum. Writing
\[
 A_m^2=\norm{e_i^0}_{M_i^0}^2+\sum_{n=1}^m\Delta t_n w_n(F_h^n)^2,
 \qquad B_m=\sum_{n=1}^m\Delta t_n\sqrt{w_n}\norm{g_{i,n}}_{\mathsf G_i^n},
\]
we obtain $U_m^2\le A_m^2+2U_mB_m$, hence $U_m\le A_m+2B_m$. Consequently,
\begin{equation}\label{eq:fullprojectedbound}
 \max_n\norm{e_i^{h,n}}_{\mathsf G_i^n}
 \le G_i(T)^{1/2}\left[\norm{e_i^0}_{M_i^0}
       +\left(\sum_n\Delta t_n(F_h^n)^2\right)^{1/2}
       +2\sum_n\Delta t_n\norm{g_{i,n}}_{\mathsf G_i^n}\right].
\end{equation}
The projected initial error vanishes. By \eqref{eq:resB}, the spatial term is at most $C\sqrt T h^{k+1}\mathcal S$; by \eqref{eq:defectbound}, the time term is at most $CT\Delta t_{\max}\mathcal T_i$. Add the nodal projection error and use uniform energy equivalence to obtain \eqref{eq:fullrate}. No discrete Gronwall estimate is needed.
\end{proof}

\begin{corollary}[Model 3]\label{cor:full3}
For variant B in model 3 the same common space--time rate holds for $(\mathbf{E},\mathbf{H},\mathbf{Q},\mathbf{V})$ if the reduced exact state $X_3=(\mathbf{E},\mathbf{H},\mathbf{Q},\mathbf{V})$ is piecewise $W^{2,\infty}(I_j;L^2)$ and satisfies the spatial assumptions. It extends to $\mathbf{P},\mathbf{D}$ if $\mathbf{P}$ has the corresponding spatial approximation regularity and is piecewise $W^{2,\infty}(I_j;L^2)$.
\end{corollary}
\begin{proof}
Subtracting the backward Euler polarization equation from its displacement equation gives, exactly,
\[
 \eps\frac{\mathbf{E}_h^n-\mathbf{E}_h^{n-1}}{\Delta t_n}
       -\curl_h \mathbf{H}_h^n+s^nV_h^n=0
\]
in the HDG weak form. Thus the reduced variables form a closed system with the coercive energy \eqref{eq:energy3}; the preceding projected-exact-solution argument applies to that system. Its electric equation has no $\partial_t s$ term, because $\mathbf{P}$ is an independent state. Finally the full polarization error satisfies
\[
 \mathbf{e}_P^{h,n}-\mathbf{e}_P^{h,n-1}=\Delta t_n(s^n \mathbf{e}_V^{h,n}+\Pi_E d_{\mathbf{P},n}),
\]
with zero initial error and unchanged event states. Summing and applying \eqref{eq:defectbound} gives
\[
 \max_n\norm{\mathbf{e}_P^{h,n}}\le s^*T\max_n\norm{\mathbf{e}_V^{h,n}}
       +CT\Delta t_{\max}\max_j\norm{\partial_t^2\mathbf{P}}_{L^\infty(I_j;L^2)}.
\]
Add the projection error and reconstruct $\mathbf{D}=\eps \mathbf{E}+\mathbf{P}$.
\end{proof}

\subsection{Implicit midpoint: fully discrete error}\label{sec:midpoint}
Here the spatial mesh is fixed when establishing second-order time accuracy. Let $Y_{i,h}(t)$ be the exact semidiscrete solution with projected initial data, and $Y_{i,h}^n$ the midpoint solution on the same mesh. On each smooth slab assume $Y_{i,h}$ has three bounded time derivatives and $\partial_t^2Y_{i,h}$ belongs to the domain of every frozen $A_{i,h}(a)$ in that slab. Define
\begin{equation}\label{eq:midregularity}
 \mathcal T_{i,h}^{\rm mid}=\max_j\mathop{\rm ess\,sup}_{a,r\in I_j}
 \left(\norm{\partial_t^3Y_{i,h}(r)}_{M_h(a)}
 +\norm{A_{i,h}(a)\partial_t^2Y_{i,h}(r)}_{M_h(a)}\right).
\end{equation}
For a fixed finite-dimensional system this quantity is finite under the stated smoothness; no $h$-independent bound is assumed. For model 3, use the reduced semidiscrete state in \eqref{eq:midregularity} and additionally assume the midpoint residual of its independent polarization equation is bounded by $\Delta t_n^2\mathcal T_{P,h}^{\rm mid}$ in $L^2$.

\begin{theorem}[Fixed-mesh fully discrete midpoint estimate]\label{thm:midfull}
Use variant B, exact event transfer, and the spatial assumptions of Theorem~\ref{thm:error}. For models 1, 2, and 4, the midpoint solution satisfies
\begin{equation}\label{eq:midfullrate}
 \max_n\norm{Z_i(t_n)-Z_{i,h}^n}_{L^2(\Omega)}
 \le C h^{k+1}\bigl(\mathcal A_i+G_i(T)^{1/2}\mathcal B\bigr)
 +C G_i(T)^{1/2}T\,\Delta t^2\mathcal T_{i,h}^{\rm mid},
 \qquad \Delta t=\max_n \Delta t_n.
\end{equation}
The same rate holds for model 3's reduced state; it extends to $\mathbf P,\mathbf D$ with an additional $C T\Delta t^2\mathcal T_{P,h}^{\rm mid}$ term and the corresponding spatial projection bound. Constants outside the displayed regularity quantities are independent of the time partition. A mesh-uniform second-order bound follows if these semidiscrete regularity quantities are bounded independently of $h$; that bound is not established here.
\end{theorem}
\begin{proof}
For an event-free step put $a=t_{n-1/2}$ and let $\bar Y_{i,h}(t_n)=\tfrac12(Y_{i,h}(t_n)+Y_{i,h}(t_{n-1}))$. The exact semidiscrete state has the midpoint defect
\[
 d_n^{\rm mid}=\frac{Y_{i,h}(t_n)-Y_{i,h}(t_{n-1})}{\Delta t_n}
                  -A_{i,h}(a)\bar Y_{i,h}(t_n).
\]
Since $\partial_tY_{i,h}(a)=A_{i,h}(a)Y_{i,h}(a)$, centered Taylor estimates give
\[
 \norm{d_n^{\rm mid}}_{M_h(a)}
 \le C\Delta t_n^2\mathcal T_{i,h}^{\rm mid}.
\]
Indeed the centered difference of $Y_{i,h}$ differs from $\partial_tY_{i,h}(a)$ by $O(\Delta t_n^2\partial_t^3Y_{i,h})$, while $\bar Y_{i,h}(t_n)-Y_{i,h}(a)=O(\Delta t_n^2\partial_t^2Y_{i,h})$ in the graph norm of $A_{i,h}(a)$.
The exact and discrete equations evaluate the same coefficients at $a$; their time variation is thus included in these temporal and graph-norm regularity quantities, while the varying energy weight is handled by the metric comparisons below.

Set $e^n=Y_{i,h}(t_n)-Y_{i,h}^n$. Subtract the two midpoint equations, test with $\bar e^n=(e^n+e^{n-1})/2$ in the frozen energy, and use \eqref{eq:frozen}:
\begin{equation}\label{eq:miderrorenergy}
 \Et_n^*(e^n)-\Et_n^*(e^{n-1})
 +\Delta t_n\bigl(\Dh(\bar e^n;a)+\mathcal L_h(\bar e^n;a)\bigr)
 =\Delta t_n(\bar e^n,d_n^{\rm mid})_{M_h(a)}.
\end{equation}
Let $r_n$ compare consecutive frozen energy metrics, including a zero-length comparison at every event, and set $w_n=\prod_{\ell\le n}\max\{1,r_\ell\}^{-1}$. As in the backward Euler energy estimate above, $G_i(T)^{-1}\le w_n\le1$. Multiplying \eqref{eq:miderrorenergy} by $w_n$, summing, and setting $U_m=\max_{n\le m}\sqrt{w_n}\norm{e^n}_{\mathsf G_n^*}$ yields
\[
 U_m^2\le 2U_m\sum_{n\le m}\Delta t_n\sqrt{w_n}\norm{d_n^{\rm mid}}_{\mathsf G_n^*}.
\]
Here the incoming half of $\bar e^n$ is bounded in the current metric by the preceding weighted energy. Thus $\max_n\norm{e^n}_{\mathsf G_n^*}\le C G_i(T)^{1/2}T\Delta t^2\mathcal T_{i,h}^{\rm mid}$. Add the semidiscrete spatial estimate \eqref{eq:spatialerror}. For model 3 the same energy argument applies to the reduced state; summing the polarization midpoint error equation bounds $\mathbf P$, and $\mathbf D=\eps\mathbf E+\mathbf P$ then follows.
\end{proof}

\paragraph{Additive forcing and manufactured solutions.}\label{sec:forcing}
In continuous-state coordinates consider the forced system
\begin{equation}\label{eq:forced}
 \partial_tY_i=\mathcal A_i(t)Y_i+F_i,\qquad
 \partial_tY_{i,h}=A_{i,h}(t)Y_{i,h}+\Pi_iF_i.
\end{equation}
Here $\mathcal A_i(t)$ is the first-order Maxwell--material operator, and $F_i$ has one source for every differential state. In particular the $\mathbf{H}$ source is in the equation for $\partial_t\mathbf{H}$; its weak load is multiplied by $\muo$. In model 3 the material equations become
\[
 \partial_t\mathbf{P}=s\mathbf{V}+F_P,\qquad \partial_t\mathbf{Q}=\mathbf{V}+F_Q,\qquad
 \partial_t\mathbf{V}+\gamma \mathbf{V}+b\mathbf{Q}=\eps \mathbf{E}+F_V.
\]
The reduced electric source is $(F_D-F_P)/\eps$. Manufactured coordinates obey these forced equations, so derivative identities such as $\partial_t\mathbf{Q}=\mathbf{V}$ or $\mathbf{J}=\partial_t\mathbf{P}$ need not hold without their source corrections. The constitutive reconstruction of $\mathbf{E}$ is unchanged. Sources are ordinary slabwise functions without event impulses; the continuous states still transfer identically.

\begin{corollary}[Exact Galerkin loads]\label{cor:forced}
Under the respective regularity assumptions, the spatial estimates of Theorem~\ref{thm:error} and Corollary~\ref{cor:baseline}, the backward Euler estimates of Theorem~\ref{thm:full} and Corollary~\ref{cor:full3}, and the midpoint estimate of Theorem~\ref{thm:midfull} hold for \eqref{eq:forced} with exact Galerkin loads. Backward Euler evaluates $\Pi_iF_i$ at the appropriate one-sided endpoint; midpoint evaluates it at the step midpoint. The stability identity \eqref{eq:commonenergy} gains the corresponding source-work term.
\end{corollary}
\begin{proof}
Projection of the continuous forced equations produces exactly the discrete load. Subtraction cancels the source in the spatial error equation and leaves the same temporal defects. For midpoint, the source at the slab midpoint also cancels when the exact semidiscrete equation is evaluated there. Model 3's polarization difference has the same cancellation. The homogeneous energy identity becomes a forced energy inequality through Cauchy--Schwarz on the source work.
\end{proof}
Numerically approximated loads leave an additional load residual, which is not included in the exact-load estimates above.

Event alignment is essential to the smooth step defects used above. If a step crosses a physical jump, its represented schedule and timing error must be analyzed separately; see Section~\ref{sec:timing}.

\section{Switch-time perturbations and convergence reduction}\label{sec:timing}
Here the coefficient histories are piecewise constant in time, with a fixed number of switches and fixed final time $T$. A represented schedule uses
\begin{equation}\label{eq:shiftdef}
 \widetilde\theta_j=\theta_j+\delta_j,\qquad
 \eta=\sum_{j=1}^{N_{\rm ev}}|\delta_j|.
\end{equation}
The coefficient values and their order are unchanged. Assume the intervals $W_j$ between $\theta_j$ and $\widetilde\theta_j$ are disjoint and the shifted events remain ordered in $(0,T)$. The histories differ only on these intervals and transfer the same continuous states. The next three results address, in order, the state error caused by moving an event, whether this timing error can reduce midpoint accuracy, and the field norm affected by the model-2 jump. A computed solution on a displaced schedule is compared with the physical solution by separating its discretization error against the exact shifted solution from the shift error between the exact shifted and physical solutions.

\subsection{State error caused by switch displacement}
Use the full continuous state $Y_i$, including $\mathbf{P}$ for model 3, and let $A_j$ be the frozen generator between successive events. The norm is the full-state $L^2$ norm, or its fixed finite-dimensional counterpart for a Fourier mode.

\begin{proposition}[Error caused by displaced switches]\label{prop:timing}
Let $Y$ and $\widetilde Y$ solve the physical and shifted systems with identical initial continuous states. Suppose the shifted propagators satisfy
\[
 \norm{\widetilde U(t,r)}\le C_U,\qquad 0\le r\le t\le T.
\]
For a continuous PDE, let $A_{j-1}$ and $A_j$ be the frozen Maxwell--material operators on the full-state $L^2$ space. Require, for almost every $r\in W_j$, that $Y(r)\in\mathcal D(A_{j-1})\cap\mathcal D(A_j)$, and that $(A_{j-1}-A_j)Y$ is strongly measurable and belongs to $L^1(W_j;L^2(\Omega))$. These common-domain and integrability assumptions are automatic for the finite-dimensional modal systems. Define
\begin{equation}\label{eq:Btiming}
 B_j=\mathop{\rm ess\,sup}_{r\in W_j}
        \norm{(A_{j-1}-A_j)Y(r)},
\end{equation}
and suppose $B_j<\infty$. An empty interval contributes zero. Then
\begin{equation}\label{eq:timingbound}
 \sup_{0\le t\le T}\norm{\widetilde Y(t)-Y(t)}
 \le C_U\sum_{j=1}^{N_{\rm ev}} |\delta_j| B_j.
\end{equation}
For HDG the bound is uniform in $h$ only if $C_U$ and $B_j$ are uniform in $h$.
\end{proposition}
\begin{proof}
Variation of constants for the difference of the two continuous states gives an integral supported only on the $W_j$. On $W_j$ its forcing is $\pm(A_{j-1}-A_j)Y$; the propagator bound and the interval length give \eqref{eq:timingbound}.
\end{proof}

\begin{remark}[Which constants are uniform over an HDG mesh?]\label{rem:timinguniform}
The energy bounds give $C_U$ independently of $h$ under the stated coefficient assumptions. For models 1, 3, and 4, switching changes only bounded material terms, so their generator defects are controlled by the state norm. In model 2, $\mathbf{E}_s=(\mathbf{D}-s\mathbf{Q})/\eps$ determines the PEC domain of the frozen Maxwell operator. A sufficient condition for both one-sided reconstructions to belong to $H_0(\curl;\Omega)$ is $(s^+-s^-)\mathbf{Q}(r)\in H_0(\curl;\Omega)$ on each swept interval, together with admissibility of one reconstruction. Broken-curl $L^2$ bounds alone do not impose the required tangential continuity or PEC trace. The magnetic generator defect includes $\curl((s^+-s^-)\mathbf{Q})/(\muo\eps)$, so full-state $L^2$ control does not itself give an $h$-uniform $B_j$. The claim for model 2 therefore requires the common-domain and uniform defect bounds; fixed bounded Fourier modes satisfy the finite-dimensional version.
\end{remark}

Let $Y_{i,h,\Delta t}^{\rm rnd}$ denote the computed state on the displaced schedule and $\widetilde Y_i$ the exact solution on that schedule. Its error against the physical solution separates as
\[
 \sup_t\norm{Y_{i,h,\Delta t}^{\rm rnd}(t)-Y_i(t)}
 \le \sup_t\norm{Y_{i,h,\Delta t}^{\rm rnd}(t)-\widetilde Y_i(t)}
       +\sup_t\norm{\widetilde Y_i(t)-Y_i(t)}.
\]
For a joint limit in $h$, $\Delta t$, and the event displacements, assume the shifted exact solutions form a uniformly regular family: the spatial and temporal regularity constants in Theorems~\ref{thm:error}, \ref{thm:full}, and \ref{thm:midfull}, as well as $C_U$ and the defect bounds $B_j$, are independent of the shifts. In model 2, compatibility must also hold at each shifted event; for example, $\llbracket(s^+-s^-)\widetilde{\mathbf{Q}}(\widetilde\theta_j)\rrbracket_T=0$ on material interfaces, with the corresponding PEC trace condition. Compatibility only at the original event does not imply compatibility after displacement. Under these assumptions, backward Euler has state error bounded by $C h^{k+1}+C\Delta t+C_U\sum_j|\delta_j|B_j$. For midpoint on a fixed mesh, the corresponding bound is $C h^{k+1}+C G_i(T)^{1/2}T\Delta t^2\mathcal T_{i,h}^{\rm mid}+C_U\sum_j|\delta_j|B_j$, with the additional model-3 polarization term specified in Theorem~\ref{thm:midfull}. Here $\Delta t$ is the maximum step. Exact event fitting makes every $\delta_j$ zero. In model 2 these state bounds control the reconstructed electric field only at times when the two coefficient histories agree.

For one switch at $\theta\in(0,T)$ in a fixed Fourier mode, let $A_-,A_+$ be the frozen continuous-state matrices. This finite-dimensional calculation identifies the leading timing term and compares it with midpoint integration error. The exact map is
\begin{equation}\label{eq:periodDL}
 S(\theta)=e^{(T-\theta)A_+}e^{\theta A_-}.
\end{equation}
The map acts in one continuous-state coordinate system; a return to the initial medium at $T$ copies that state.

\begin{proposition}[First variation and midpoint map error]\label{prop:midtiming}
Fix the matrices, or let them range in a bounded smooth parameter family with fixed positive dwell times. Let $S_{\Delta t}^{\rm fit}$ be the midpoint product on a grid split at the physical switch, and $S_{\Delta t}^{\rm rnd}$ the midpoint product on a grid fitted to $\theta+\delta$, both with maximum step at most $\Delta t$ and final time $T$. Then
\begin{align}
 K(\theta)&=e^{(T-\theta)A_+}(A_--A_+)e^{\theta A_-},\label{eq:Ktiming}\\
 S_{\Delta t}^{\rm rnd}-S(\theta)&=\delta K(\theta)+O(\Delta t^2+\delta^2),\qquad
 S_{\Delta t}^{\rm fit}-S(\theta)=O(\Delta t^2).\label{eq:fitmidtiming}
\end{align}
Thus, when $|\delta|$ is comparable to $\Delta t$ and $K(\theta)Y_0\ne0$, the first-order timing term dominates midpoint's second-order integration error.
\end{proposition}
\begin{proof}
Differentiating \eqref{eq:periodDL} gives $S'(\theta)=K(\theta)$. On a constant slab a midpoint step differs from its matrix exponential by $O(\Delta t_n^3)$; summing these defects gives $O(\Delta t^2)$. Taylor expansion in $\delta$ then yields the stated estimates.
\end{proof}

\subsection{Sharpness for a spatially exact Lorentz mode}
\begin{proposition}[First-order switch-shifted-switch midpoint]\label{prop:sharpDL}
On a unit-volume periodic cell choose a fixed unit vector $\mathbf{u}$ and take spatially constant fields with $\mathbf{H}=0$. For model 1 with $\eps=\muo=1$, $b=1$, and $\gamma=0$, prescribe
\[
 \mathbf{D}(0)=\mathbf{u},\quad \mathbf{P}(0)=\mathbf{J}(0)=0,\quad
 s(t)=\begin{cases}1,&t<1/3,\\2,&t>1/3.\end{cases}
\]
Fix $T=1$, $\theta=1/3$, and uniform base grids with
\[
N_r=3\cdot2^r+1,\qquad \Delta t_r=1/N_r,\qquad
 \widetilde\theta_r=2^r\Delta t_r,\quad r\ge1.
\]
Since $N_r/3=2^r+1/3$, the selected grid node is the one with index $2^r$, and $\delta_r=\widetilde\theta_r-\theta=-\Delta t_r/3$. There are constants $c,C>0$, independent of $r$, such that for all sufficiently large $r$ each of the following final continuous-state errors lies between $c\Delta t_r$ and $C\Delta t_r$: the exact propagator on this displaced schedule, and midpoint with coefficients sampled at the base-step midpoints. Midpoint with the crossing step split at the exact switch has error $O(\Delta t_r^2)$. Both HDG variants represent this spatially constant evolution exactly in space, for every $k\ge0$.
\end{proposition}
\begin{proof}
Write the scalar amplitudes in the $\mathbf{u}$ direction as
$\mathbf{D}=d\mathbf{u}$, $\mathbf{P}=p\mathbf{u}$,
$\mathbf{J}=j\mathbf{u}$, and $\mathbf{E}=e\mathbf{u}$.  The invariant
subsystem in $y=(d,p,j)^T$ is
\begin{equation}\label{eq:sharpA}
 \partial_t y=A(s)y,\qquad
 A(s)=\begin{pmatrix}0&0&0\\0&0&1\\s&-(s+1)&0\end{pmatrix},
 \quad y_0=(1,0,0)^T.
\end{equation}
Before the switch the scalar electric amplitude is $e(t)=\tfrac12(1+\cos(\sqrt2 t))$, so $e(\theta)>0$. Hence
\[
 (A_--A_+)y(\theta)=(0,0,-e(\theta))^T\ne0.
\]
The exponential $e^{(T-\theta)A_+}$ is invertible, so $v=K(\theta)y_0\ne0$. Thus
\[
 y_{\rm rnd}^{\rm exact}(T)-y(T)=\delta_r v+O(\Delta t_r^2),
 \qquad y_{\rm rnd}^{\rm MP}(T)-y(T)=\delta_r v+O(\Delta t_r^2).
\]
Their norms are $\Delta t_r\norm v/3+O(\Delta t_r^2)$, which proves both upper and lower first-order bounds. Since $N_r/3=2^r+1/3$, the crossing-step midpoint lies after the physical switch. Sampling therefore treats that whole step as post-switch, exactly as switching at its earlier node $\widetilde\theta_r$. The fitted estimate follows from \eqref{eq:fitmidtiming}. Finally all element fields and facet traces are constant-compatible, their spatial curls vanish, and $\mathbf{E}_{h,T}=\widehat{\mathbf{E}}_h$. Hence the HDG dissipation is zero and the HDG equations reduce exactly to \eqref{eq:sharpA}.
\end{proof}

Thus even exact integration on the displaced schedule has first-order error relative to the fixed physical problem; the midpoint loss comes from the switch time representation.

\subsection{Electric-field error in model 2}
\begin{proposition}[State convergence need not give uniform field convergence]\label{prop:m2pulse}
Consider a homogeneous finite-dimensional model-2 evolution with a single change $s_-\ne s_+$ at $\theta$, and its exact shifted-switch evolution at $\theta+\delta$. Let their initial continuous states agree and suppose $\mathbf{Q}(\theta)\ne0$. Normalize the spatial mode to unit $L^2$ norm and set
\[
a=\frac{|s_+-s_-|}{\eps}\norm{\mathbf{Q}(\theta)}>0.
\]
As $\delta\to0$, the continuous states have uniform error $O(|\delta|)$, but their algebraically reconstructed electric fields satisfy
\begin{align}
 \norm{\widetilde{\mathbf{E}}-\mathbf{E}}_{L^\infty(0,T;L^2)}&=a+O(|\delta|),\label{eq:m2sup}\\
 \norm{\widetilde{\mathbf{E}}-\mathbf{E}}_{L^2(0,T;L^2)}^2&=a^2|\delta|+O(\delta^2).\label{eq:m2L2}
\end{align}
In particular the latter field norm is bounded above and below by positive constant multiples of $|\delta|^{1/2}$, whereas the final-time field error is $O(|\delta|)$ if $T$ lies after both switches. The polarization $\mathbf{P}=s\mathbf{Q}$ has the same phenomenon, with jump amplitude $|s_+-s_-|\norm{\mathbf{Q}(\theta)}$.
\end{proposition}
\begin{proof}
Bounded modal generators give the uniform continuous-state estimate by Proposition~\ref{prop:timing}. The constitutive law gives
\[
 \widetilde{\mathbf{E}}-\mathbf{E}
 =\frac{\widetilde{\mathbf{D}}-\mathbf{D}-\widetilde{s}(\widetilde{\mathbf{Q}}-\mathbf{Q})}{\eps}
  -\frac{(\widetilde{s}-s)\mathbf{Q}}{\eps}.
\]
The first term is $O(|\delta|)$. The second vanishes off the interval swept by the switch and has height approaching the constitutive jump on that interval. There,
\[
 \widetilde{\mathbf{E}}-\mathbf{E}
 =\pm\frac{s_+-s_-}{\eps}\mathbf{Q}(\theta)+O(|\delta|),
\]
where continuity and piecewise bounded derivatives of $\mathbf{Q}$ give a uniform remainder. Taking the essential supremum gives \eqref{eq:m2sup}; integrating the squared norm over the swept interval and adding the $O(\delta^2)$ contribution outside gives \eqref{eq:m2L2}. The argument for $\mathbf{P}$ uses its constitutive reconstruction.
\end{proof}

Consequently, the state timing estimate \eqref{eq:timingbound} does not imply uniform-in-time electric-field convergence for model 2 on mismatched schedules. This pulse persists even when both histories are integrated exactly; event fitting removes it.

\section{Numerical experiments}\label{sec:numerical}
The experiments test five claims: spatial convergence, fixed-mesh time convergence, loss of midpoint order when a switch is displaced, the model-2 electric-field jump, and the resulting change in stability classification. The three-dimensional HDG tests use $\eps=\muo=1$ on a unit cube split into $6n^3$ Kuhn tetrahedra, so $h=\sqrt3/n$. Variant A uses $\tau=1$ and variant B uses $\tau=h^{-1}$, both with full tangential traces. Material coefficients are constant on each element.

For a model's continuous state $Y_i$, distinguish two error norms:
\begin{align}
 e_{i,\mathrm{state}}^2(t)&=\sum_{\mathbf{Z}\in Y_i}\|\mathbf{Z}(t)-\mathbf{Z}_h(t)\|_{L^2(\Omega)}^2,\label{eq:stateerrnumeric}\\
 e_{i,\mathrm{field+state}}^2(t)&=e_{i,\mathrm{state}}^2(t)+\|\mathbf{E}(t)-\mathbf{E}_h(t)\|_{L^2(\Omega)}^2.\label{eq:commonerrnumeric}
\end{align}
Both sums retain the independent $\mathbf{P}$ state of model 3. Spatial and time convergence tests use $e_{i,\mathrm{field+state}}$; switch-timing comparisons use $e_{i,\mathrm{state}}$. This distinction matters for model 2 when the switching histories differ.

\subsection{Spatial convergence}\label{sec:numconvergence}
The first test asks whether the three-dimensional HDG error decreases at the rates proved in Section~\ref{sec:error}. With PEC boundaries, set
\[
 \mathbf{e}(x,y,z)=\begin{pmatrix}
 \sin(\pi y)\sin(\pi z)\\
 \sin(\pi z)\sin(\pi x)\\
 \sin(\pi x)\sin(\pi y)
 \end{pmatrix},\qquad \mathbf{h}_c=\curl \mathbf{e},\qquad \curl \mathbf{h}_c=2\pi^2\mathbf{e}.
\]
This is the manufactured-solution experiment. Each electric or material state is a scalar multiple of $\mathbf e$, and $\mathbf H$ is a scalar multiple of $\mathbf h_c$. In the state order of Section~\ref{sec:models}, the prescribed amplitudes are $a\cos(1.7t)+d\sin(1.7t)$, with $a=(1.1,0.17,0.23,0.31,0.19)$ and $d=(0.21,-0.13,0.11,-0.17,0.07)$ truncated to the state dimension. With fixed $s=2$, $b=3$, and $\gamma=0.1$, substitution determines the forcing in every first-order equation as in \eqref{eq:forced}. This prescribed field is the PDE reference solution. Model 3 retains independently prescribed $\mathbf P,\mathbf Q,\mathbf V$.

Matrix-exponential evolution removes time-step error. Tables~\ref{tab:spaceA0}--\ref{tab:spaceB1} report the maximum field-plus-state error at seven times in $[0,0.3]$ for $n=2,4,8$; each refinement halves $h$, and rates are $\log_2(e_{\rm coarse}/e_{\rm fine})$.

The last-grid rates for variant B are $1.11$--$1.12$ when $k=0$ and $2.15$--$2.17$ when $k=1$, consistent with the proved $h^{k+1}$ common-state rate. Variant A approaches orders $0.9$ and $1.9$ on this smooth example; its proof only guarantees $h^{k+1/2}$. These measurements do not establish an additional electric-field order.

\begin{table}[htbp]\centering\small
\caption{Spatial refinement: variant A, $k=0$. Each entry is the maximum sampled field-plus-state error, followed by the adjacent-grid rate.}
\label{tab:spaceA0}
\begin{tabular}{crrrr}\toprule
$n$ & Model 1 & Model 2 & Model 3 & Model 4\\\midrule
2 & $7.072\mathrm{e}{-1}$ (--) & $5.989\mathrm{e}{-1}$ (--) & $7.040\mathrm{e}{-1}$ (--) & $7.001\mathrm{e}{-1}$ (--)\\
4 & $4.143\mathrm{e}{-1}$ (0.771) & $3.386\mathrm{e}{-1}$ (0.823) & $4.104\mathrm{e}{-1}$ (0.779) & $4.086\mathrm{e}{-1}$ (0.777)\\
8 & $2.237\mathrm{e}{-1}$ (0.889) & $1.804\mathrm{e}{-1}$ (0.908) & $2.207\mathrm{e}{-1}$ (0.895) & $2.198\mathrm{e}{-1}$ (0.894)\\
\bottomrule\end{tabular}\end{table}
\begin{table}[htbp]\centering\small
\caption{Spatial refinement: variant A, $k=1$. Each entry is the maximum sampled field-plus-state error, followed by the adjacent-grid rate.}
\label{tab:spaceA1}
\begin{tabular}{crrrr}\toprule
$n$ & Model 1 & Model 2 & Model 3 & Model 4\\\midrule
2 & $2.088\mathrm{e}{-1}$ (--) & $1.749\mathrm{e}{-1}$ (--) & $2.079\mathrm{e}{-1}$ (--) & $2.068\mathrm{e}{-1}$ (--)\\
4 & $6.207\mathrm{e}{-2}$ (1.750) & $5.133\mathrm{e}{-2}$ (1.769) & $6.199\mathrm{e}{-2}$ (1.746) & $6.174\mathrm{e}{-2}$ (1.744)\\
8 & $1.648\mathrm{e}{-2}$ (1.913) & $1.379\mathrm{e}{-2}$ (1.896) & $1.652\mathrm{e}{-2}$ (1.908) & $1.646\mathrm{e}{-2}$ (1.907)\\
\bottomrule\end{tabular}\end{table}
\begin{table}[htbp]\centering\small
\caption{Spatial refinement: variant B, $k=0$. Each entry is the maximum sampled field-plus-state error, followed by the adjacent-grid rate.}
\label{tab:spaceB0}
\begin{tabular}{crrrr}\toprule
$n$ & Model 1 & Model 2 & Model 3 & Model 4\\\midrule
2 & $3.819\mathrm{e}{-1}$ (--) & $3.497\mathrm{e}{-1}$ (--) & $3.777\mathrm{e}{-1}$ (--) & $3.771\mathrm{e}{-1}$ (--)\\
4 & $1.732\mathrm{e}{-1}$ (1.141) & $1.673\mathrm{e}{-1}$ (1.064) & $1.731\mathrm{e}{-1}$ (1.126) & $1.730\mathrm{e}{-1}$ (1.124)\\
8 & $7.953\mathrm{e}{-2}$ (1.122) & $7.741\mathrm{e}{-2}$ (1.112) & $7.950\mathrm{e}{-2}$ (1.122) & $7.949\mathrm{e}{-2}$ (1.122)\\
\bottomrule\end{tabular}\end{table}
\begin{table}[htbp]\centering\small
\caption{Spatial refinement: variant B, $k=1$. Each entry is the maximum sampled field-plus-state error, followed by the adjacent-grid rate.}
\label{tab:spaceB1}
\begin{tabular}{crrrr}\toprule
$n$ & Model 1 & Model 2 & Model 3 & Model 4\\\midrule
2 & $1.121\mathrm{e}{-1}$ (--) & $1.070\mathrm{e}{-1}$ (--) & $1.114\mathrm{e}{-1}$ (--) & $1.113\mathrm{e}{-1}$ (--)\\
4 & $2.565\mathrm{e}{-2}$ (2.128) & $2.505\mathrm{e}{-2}$ (2.095) & $2.561\mathrm{e}{-2}$ (2.121) & $2.560\mathrm{e}{-2}$ (2.120)\\
8 & $5.714\mathrm{e}{-3}$ (2.166) & $5.640\mathrm{e}{-3}$ (2.151) & $5.713\mathrm{e}{-3}$ (2.164) & $5.712\mathrm{e}{-3}$ (2.164)\\
\bottomrule\end{tabular}\end{table}

\subsection{Temporal convergence with smooth and switching coefficients}\label{sec:numtimeconvergence}
The second test checks first-order backward Euler and second-order midpoint convergence on the same fixed HDG mesh ($n=2$, variant B, $k=0$, $T=0.4$). This is an unforced evolution, not the manufactured solution above: only its cavity shapes supply the initial state. Both methods start from the normalized block projection of those shapes with amplitudes $(1.1,0.17,0.23,0.31,0.19)$, truncated to the state dimension. The smooth coefficients are
\[
 s(t)=2+0.4\sin(3t),\qquad b(t)=3+0.2\cos(2t),\qquad\gamma=0.15.
\]
The switching test adds $0.8$ to $s$ and $0.3$ to $b$ at $t=0.17$, transferring the continuous state unchanged. For $N=10,20,40,80,160$, each event-free slab is divided into equal steps no longer than $0.4/N$; the event is always a node. Backward Euler uses endpoint coefficients and midpoint uses midpoint coefficients.

The reference is a high-accuracy DOP853 solution of this unforced semidiscrete HDG system with the same coefficient history, integrated separately on each event-free slab. No right-hand side is recomputed to preserve the amplitudes used in the spatial test. The reference tolerance check changes the smallest reported error by less than $9.6\mathrm{e}{-6}$ relatively. Tables~\ref{tab:timesmooth} and \ref{tab:timeevent} show final errors and adjacent-step rates; at the event, electric fields are compared on matching sides.

Across all four models and both histories, the finest backward Euler rates are $0.981$--$0.987$, while midpoint rates round to $2.000$. This verifies the expected orders at the fixed mesh, including a correctly fitted event; it does not establish a mesh-uniform midpoint constant.

\begin{table}[htbp]\centering\small
\caption{Temporal refinement (smooth history): final field-plus-state errors and adjacent-grid rates against the same semidiscrete reference. $N$ sets the target step $0.4/N$; actual slab steps determine the rates.}
\label{tab:timesmooth}
\begin{tabular}{crrrrrr}\toprule
Model & $N$ & BE error & Rate & Midpoint error & Rate & MP node maximum\\\midrule
1 & 10 & $1.121\mathrm{e}{-1}$ & -- & $5.470\mathrm{e}{-3}$ & -- & $5.470\mathrm{e}{-3}$\\
1 & 20 & $5.988\mathrm{e}{-2}$ & 0.904 & $1.365\mathrm{e}{-3}$ & 2.002 & $1.365\mathrm{e}{-3}$\\
1 & 40 & $3.101\mathrm{e}{-2}$ & 0.950 & $3.412\mathrm{e}{-4}$ & 2.000 & $3.412\mathrm{e}{-4}$\\
1 & 80 & $1.579\mathrm{e}{-2}$ & 0.974 & $8.530\mathrm{e}{-5}$ & 2.000 & $8.530\mathrm{e}{-5}$\\
1 & 160 & $7.967\mathrm{e}{-3}$ & 0.987 & $2.133\mathrm{e}{-5}$ & 2.000 & $2.133\mathrm{e}{-5}$\\
\midrule
2 & 10 & $7.575\mathrm{e}{-2}$ & -- & $4.228\mathrm{e}{-3}$ & -- & $4.412\mathrm{e}{-3}$\\
2 & 20 & $4.087\mathrm{e}{-2}$ & 0.890 & $1.053\mathrm{e}{-3}$ & 2.005 & $1.097\mathrm{e}{-3}$\\
2 & 40 & $2.130\mathrm{e}{-2}$ & 0.940 & $2.631\mathrm{e}{-4}$ & 2.001 & $2.737\mathrm{e}{-4}$\\
2 & 80 & $1.088\mathrm{e}{-2}$ & 0.969 & $6.577\mathrm{e}{-5}$ & 2.000 & $6.840\mathrm{e}{-5}$\\
2 & 160 & $5.502\mathrm{e}{-3}$ & 0.984 & $1.644\mathrm{e}{-5}$ & 2.000 & $1.710\mathrm{e}{-5}$\\
\midrule
3 & 10 & $9.675\mathrm{e}{-2}$ & -- & $5.148\mathrm{e}{-3}$ & -- & $5.148\mathrm{e}{-3}$\\
3 & 20 & $5.190\mathrm{e}{-2}$ & 0.899 & $1.285\mathrm{e}{-3}$ & 2.003 & $1.285\mathrm{e}{-3}$\\
3 & 40 & $2.695\mathrm{e}{-2}$ & 0.946 & $3.210\mathrm{e}{-4}$ & 2.001 & $3.210\mathrm{e}{-4}$\\
3 & 80 & $1.374\mathrm{e}{-2}$ & 0.972 & $8.025\mathrm{e}{-5}$ & 2.000 & $8.025\mathrm{e}{-5}$\\
3 & 160 & $6.941\mathrm{e}{-3}$ & 0.985 & $2.006\mathrm{e}{-5}$ & 2.000 & $2.006\mathrm{e}{-5}$\\
\midrule
4 & 10 & $1.002\mathrm{e}{-1}$ & -- & $5.370\mathrm{e}{-3}$ & -- & $5.370\mathrm{e}{-3}$\\
4 & 20 & $5.362\mathrm{e}{-2}$ & 0.902 & $1.340\mathrm{e}{-3}$ & 2.002 & $1.340\mathrm{e}{-3}$\\
4 & 40 & $2.781\mathrm{e}{-2}$ & 0.947 & $3.350\mathrm{e}{-4}$ & 2.001 & $3.350\mathrm{e}{-4}$\\
4 & 80 & $1.417\mathrm{e}{-2}$ & 0.973 & $8.373\mathrm{e}{-5}$ & 2.000 & $8.373\mathrm{e}{-5}$\\
4 & 160 & $7.155\mathrm{e}{-3}$ & 0.986 & $2.093\mathrm{e}{-5}$ & 2.000 & $2.093\mathrm{e}{-5}$\\
\bottomrule\end{tabular}\end{table}
\begin{table}[htbp]\centering\small
\caption{Temporal refinement (event history): final field-plus-state errors and adjacent-grid rates against the same semidiscrete reference. $N$ sets the target step $0.4/N$; actual slab steps determine the rates.}
\label{tab:timeevent}
\begin{tabular}{crrrrrr}\toprule
Model & $N$ & BE error & Rate & Midpoint error & Rate & MP node maximum\\\midrule
1 & 10 & $1.109\mathrm{e}{-1}$ & -- & $4.681\mathrm{e}{-3}$ & -- & $4.681\mathrm{e}{-3}$\\
1 & 20 & $6.126\mathrm{e}{-2}$ & 0.856 & $1.268\mathrm{e}{-3}$ & 1.884 & $1.268\mathrm{e}{-3}$\\
1 & 40 & $3.319\mathrm{e}{-2}$ & 0.942 & $3.493\mathrm{e}{-4}$ & 1.982 & $3.493\mathrm{e}{-4}$\\
1 & 80 & $1.690\mathrm{e}{-2}$ & 0.974 & $8.731\mathrm{e}{-5}$ & 2.000 & $8.731\mathrm{e}{-5}$\\
1 & 160 & $8.526\mathrm{e}{-3}$ & 0.987 & $2.183\mathrm{e}{-5}$ & 2.000 & $2.183\mathrm{e}{-5}$\\
\midrule
2 & 10 & $5.642\mathrm{e}{-2}$ & -- & $3.068\mathrm{e}{-3}$ & -- & $3.190\mathrm{e}{-3}$\\
2 & 20 & $3.210\mathrm{e}{-2}$ & 0.814 & $8.385\mathrm{e}{-4}$ & 1.871 & $9.768\mathrm{e}{-4}$\\
2 & 40 & $1.763\mathrm{e}{-2}$ & 0.921 & $2.312\mathrm{e}{-4}$ & 1.980 & $2.737\mathrm{e}{-4}$\\
2 & 80 & $9.041\mathrm{e}{-3}$ & 0.963 & $5.778\mathrm{e}{-5}$ & 2.000 & $6.840\mathrm{e}{-5}$\\
2 & 160 & $4.581\mathrm{e}{-3}$ & 0.981 & $1.444\mathrm{e}{-5}$ & 2.000 & $1.710\mathrm{e}{-5}$\\
\midrule
3 & 10 & $8.822\mathrm{e}{-2}$ & -- & $4.321\mathrm{e}{-3}$ & -- & $4.321\mathrm{e}{-3}$\\
3 & 20 & $4.904\mathrm{e}{-2}$ & 0.847 & $1.172\mathrm{e}{-3}$ & 1.883 & $1.172\mathrm{e}{-3}$\\
3 & 40 & $2.666\mathrm{e}{-2}$ & 0.937 & $3.228\mathrm{e}{-4}$ & 1.982 & $3.228\mathrm{e}{-4}$\\
3 & 80 & $1.360\mathrm{e}{-2}$ & 0.971 & $8.070\mathrm{e}{-5}$ & 2.000 & $8.070\mathrm{e}{-5}$\\
3 & 160 & $6.868\mathrm{e}{-3}$ & 0.985 & $2.017\mathrm{e}{-5}$ & 2.000 & $2.017\mathrm{e}{-5}$\\
\midrule
4 & 10 & $9.407\mathrm{e}{-2}$ & -- & $4.658\mathrm{e}{-3}$ & -- & $4.658\mathrm{e}{-3}$\\
4 & 20 & $5.207\mathrm{e}{-2}$ & 0.853 & $1.261\mathrm{e}{-3}$ & 1.886 & $1.261\mathrm{e}{-3}$\\
4 & 40 & $2.824\mathrm{e}{-2}$ & 0.940 & $3.472\mathrm{e}{-4}$ & 1.982 & $3.472\mathrm{e}{-4}$\\
4 & 80 & $1.439\mathrm{e}{-2}$ & 0.973 & $8.678\mathrm{e}{-5}$ & 2.000 & $8.678\mathrm{e}{-5}$\\
4 & 160 & $7.266\mathrm{e}{-3}$ & 0.986 & $2.169\mathrm{e}{-5}$ & 2.000 & $2.169\mathrm{e}{-5}$\\
\bottomrule\end{tabular}\end{table}

\subsection{Switch-time displacement reduces midpoint accuracy}\label{sec:numtime}
The third test isolates the error caused by representing a physical switch at the wrong time. This is the unforced, spatially exact model-1 Lorentz mode of Proposition~\ref{prop:sharpDL}, with $y_0=(1,0,0)^T$, $s:1\to2$, $\theta=1/3$, and $T=1$. Its physical reference and the exact solution of the displaced-switch problem are, respectively,
\[
 y_{\theta}(T)=e^{(T-\theta)A(2)}e^{\theta A(1)}y_0,
 \qquad
 y_{\widetilde\theta}(T)=e^{(T-\widetilde\theta)A(2)}e^{\widetilde\theta A(1)}y_0,
\]
with $A(s)$ from \eqref{eq:sharpA}. For $N=3\cdot2^r+1$, moving the event to the preceding grid node gives $|\widetilde\theta-\theta|=1/(3N)$; the fitted grid inserts $\theta$ as a node. Table~\ref{tab:numtiming} compares every final $(\mathbf D,\mathbf P,\mathbf J)$ result with $y_{\theta}(T)$. No forcing is introduced in either history.

\begin{table}[htbp]\centering
\caption{Unforced model-1 Lorentz benchmark: final error against the exact physical-switch solution. Exact evolution with the displaced switch isolates the error caused by switch-time representation.}\label{tab:numtiming}
\begin{tabular}{rrrrr}\toprule
$N$ & $\Delta t$ & Exact evolution, shifted switch & Midpoint, fitted & Midpoint, shifted switch\\\midrule
25  & $4.0000\mathrm{e}{-2}$ & $8.3482\mathrm{e}{-3}$ & $3.4136\mathrm{e}{-4}$ & $8.0409\mathrm{e}{-3}$\\
49  & $2.0408\mathrm{e}{-2}$ & $4.2681\mathrm{e}{-3}$ & $8.9978\mathrm{e}{-5}$ & $4.1884\mathrm{e}{-3}$\\
97  & $1.0309\mathrm{e}{-2}$ & $2.1583\mathrm{e}{-3}$ & $2.3105\mathrm{e}{-5}$ & $2.1380\mathrm{e}{-3}$\\
193 & $5.1813\mathrm{e}{-3}$ & $1.0853\mathrm{e}{-3}$ & $5.8548\mathrm{e}{-6}$ & $1.0802\mathrm{e}{-3}$\\
385 & $2.5974\mathrm{e}{-3}$ & $5.4423\mathrm{e}{-4}$ & $1.4736\mathrm{e}{-6}$ & $5.4294\mathrm{e}{-4}$\\\midrule
\multicolumn{2}{r}{Finest rate}&$1.000$&$1.998$&$0.996$\\\bottomrule
\end{tabular}
\end{table}

The fitted midpoint error is second order; both the exact flow and midpoint on the displaced schedule have first-order error. The constant mode is represented exactly by the three-dimensional periodic HDG space, so this loss cannot be attributed to spatial error. A propagating HDG mode on a fixed mesh gives the same distinction for all four models: approximately second order with fitting and first order with a displaced switch. Thus event alignment is necessary to retain midpoint's order relative to the physical history.

\subsection{A displaced switch prevents uniform electric-field convergence}\label{sec:numpulse}
The fourth test checks the norm distinction in Proposition~\ref{prop:m2pulse}. For model 2, use the unforced unit constant mode with initial $(\mathbf D,\mathbf H,\mathbf Q,\mathbf V)=(1,0,0.2,0.1)$, $s:1\to2$, $b=1$, $\gamma=0$, and $\theta=1/3$. Both the physical and displaced-switch histories are integrated exactly from this same initial state; only the event time differs. The state $L^\infty_t$ error and the final-time electric error are first order in $|\delta|$, while the electric $L^2_t$ error has half order. The electric $L^\infty_t$ error approaches the nonzero jump amplitude $0.264833014254$ (Figure~\ref{fig:numpulse}). Thus state convergence does not give uniform-in-time field convergence when the switching histories differ.

\begin{figure}[htbp]\centering
\includegraphics[width=.95\linewidth]{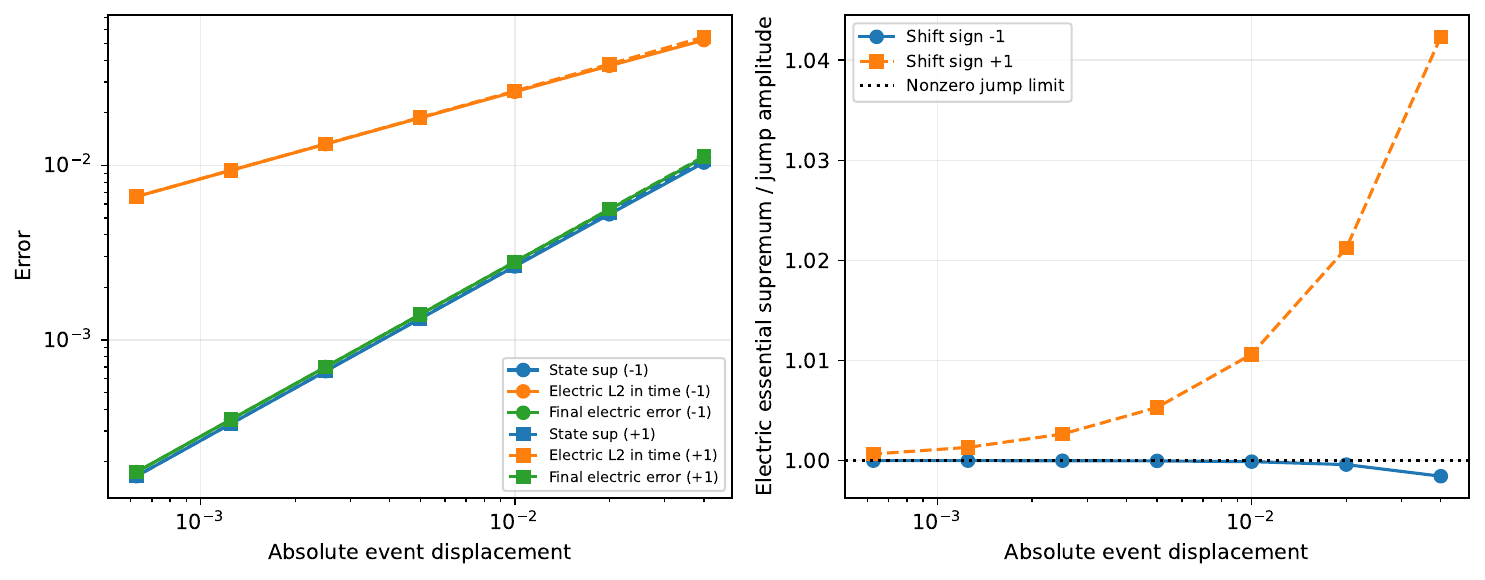}
\caption{Exact model-2 histories with displaced events. The state curve uses $\|e_{2,\mathrm{state}}\|_{L^\infty_t}$ from \eqref{eq:stateerrnumeric}, excluding the electric reconstruction. State and final electric errors vanish at first order; the time-integrated electric error has half order, and the essential supremum approaches the nonzero jump amplitude.}\label{fig:numpulse}
\end{figure}

\subsection{A displaced switch can reverse the stability prediction}\label{sec:numfloquet}
The final test asks whether timing error can change a qualitative prediction. It uses an unforced model-4 Fourier mode with wave number $1$, $s:1\to9$, $b=1$, period $T=1$, and $\theta=1/3$. In the amplitude order $(\mathbf D,\mathbf H,\mathbf P,\mathbf V)$, the frozen generators used on the two sides of the switch are
\[
 A_\pm(\gamma)=
 \begin{pmatrix}
 0&-\mathrm i&0&0\\
 -\mathrm i&0&\mathrm i&0\\
 0&0&0&s_\pm\\
 1&0&-(1+1/s_\pm)&-\gamma
 \end{pmatrix},\qquad s_-=1,\quad s_+=9.
\]
For the physical switch, the exact one-period map is $S(\theta;\gamma)=e^{(2/3)A_+(\gamma)}e^{(1/3)A_-(\gamma)}$. The critical damping solves the scalar equation $\log\rho(S(\theta;\gamma))=0$, where $\rho$ denotes spectral radius. In the root search, the simple eigenvalue near $-1$ is tracked and verified to be dominant. This gives $\gamma_*=1.522772261776$, the boundary of this fixed Fourier mode without time discretization.

For $N=385$, $\Delta t=1/385$, the displaced schedule places the switch at the preceding node $\widetilde\theta=128/385$. Its midpoint period map $S_{\Delta t}^{\rm rnd}(\gamma)$ is the ordered product of $128$ pre-switch and $257$ post-switch Cayley steps $(I-\Delta t A_\pm/2)^{-1}(I+\Delta t A_\pm/2)$. Solving $\log\rho(S_{\Delta t}^{\rm rnd}(\gamma))=0$ gives $\gamma_{\rm rnd}=1.520374893005$; solving the same equation for the fitted midpoint map $S_{\Delta t}^{\rm fit}(\gamma)$ gives $\gamma_{\rm fit}=1.522783397199$. We choose $\gamma=1.521573577390525\approx(\gamma_*+\gamma_{\rm rnd})/2$, between the physical and displaced-switch boundaries, specifically to test whether their stability classifications disagree. At this value the one-period spectral radii are
\[
 \rho(M)=1.000578448,\quad
 \rho(M_{\rm fit})=1.000583818,\quad
 \rho(M_{\rm rnd})=0.999422123.
\]
The physical mode grows, while the displaced-switch calculation predicts decay. Starting from the same unit physical dominant eigenvector, their norms after $5{,}000$ periods are $18.0186$ and $0.0555291$, respectively (Figure~\ref{fig:missed}); the fitted calculation gives $18.5088$. This fixed-mode example shows why the switch-time error matters for stability classification, without asserting a general Floquet theorem for the HDG discretization.

\begin{figure}[htbp]\centering
\includegraphics[width=.78\linewidth]{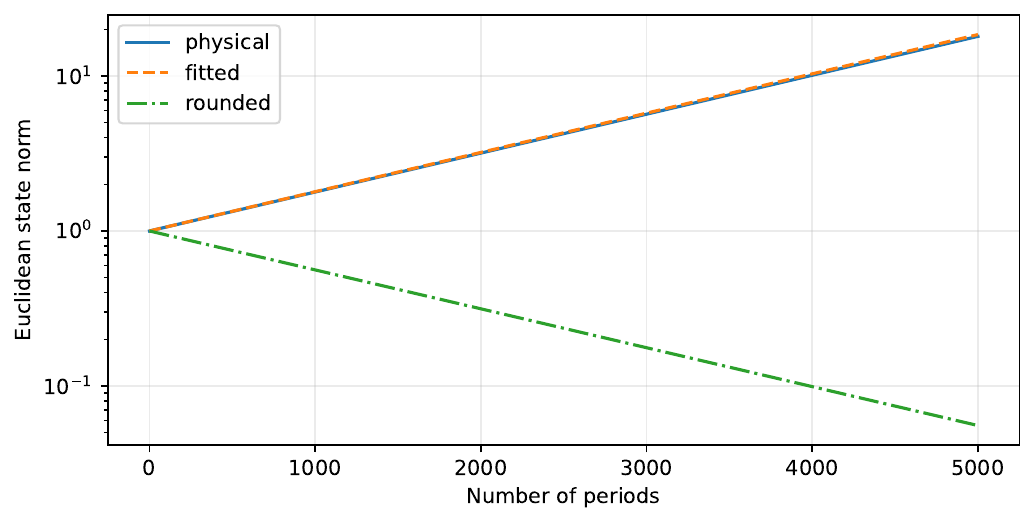}
\caption{Model 4 at a damping value between the physical and switch-shifted stability boundaries. All curves start from the same unit physical dominant eigenvector. Event fitting and switch displacement use the same base grid, $N=385$.}\label{fig:missed}
\end{figure}

\clearpage

\section{Discussion and conclusions}\label{sec:discussion}

We proposed two HDG methods (variants A and B) for four Drude--Lorentz models with time-dependent coefficients (models 1--4) and analyzed their error convergence. For every model, we established energy estimates and identified how coefficient discontinuities enter the energy balance. The proved spatial common-state convergence rates are $O(h^{k+1/2})$ for variant A and $O(h^{k+1})$ for variant B. We also analyzed the stability and error of the backward Euler and implicit midpoint time discretizations. On event-aligned grids, both methods are energy stable without a CFL restriction for either spatial variant. For variant B, backward Euler has a fully discrete $O(h^{k+1}+\Delta t)$ common-state bound, where $\Delta t$ is the maximum time step. Midpoint is second order in $\Delta t$ on each fixed mesh under the stated semidiscrete regularity assumptions. The spatial experiments show nearly $h^{k+1}$ convergence even for variant A, beyond its proved rate, and the fixed-mesh temporal experiments show the expected first and second orders.

The midpoint rate requires the physical switch times to be grid nodes: displacing an event can reduce its accuracy to first order even when the spatial solution is exact. In model 2, the displaced event also produces an electric-field pulse that prevents uniform-in-time field convergence on mismatched histories. The fixed-mode stability experiment shows that a small timing error can reverse a growth/decay prediction. A mesh-uniform second-order midpoint bound and sharper electric-field estimates remain open.

An infinite-time exponential-decay result would require additional conditions balancing dissipation against modulation and an analysis of stationary components, including the independent memory state in model 3. Such a result is distinct from the finite-time estimates proved here; see~\cite{RennerYousept2026} for a decay analysis in a different Maxwell setting.

\section*{Declaration of AI-assisted writing}
ChatGPT (OpenAI) was used to assist with drafting and editing this manuscript. The author reviewed and revised the AI-assisted text and takes full responsibility for the final content.

\end{document}